\documentclass[oneside,onecolumn,12pt]{article}
\usepackage[latin2]{inputenc}
\usepackage{amssymb}
\usepackage{paralist}
\usepackage{amsmath}
\usepackage{amsthm}
\usepackage{amsfonts}
\usepackage{graphicx}
\usepackage{wrapfig}
\usepackage{enumerate}
\usepackage{mathtools}
\usepackage{url}
\usepackage{cleveref}
\usepackage{fullpage}
\usepackage{bbm}

\newcommand{\fqt}{\zb{F}_{q^{t-1}}}
\newcommand{\fqu}{\zb{F}_q^*}
\newcommand{\fq}{\zb{F}_q}
\newcommand{\fqq}[1]{\zb{F}_{q^{#1}}}
\newcommand{\fqc}{\quer{\zb{F}}_q}

\newcommand{\gdw}{\Leftrightarrow}
\newcommand{\imp}{\Rightarrow}

\newcommand{\ohne}[1]{\backslash\men{#1}}

\newcommand{\absw}[1]{\left\{\begin{array}{ll}#1\end{array}\right.}

\newcommand{\inn}{\in\zb{N}}

\newcommand{\quer}[1]{\overline{#1}}

\newcommand{\zb}[1]{\mathbb{\uppercase{#1}}}

\newcommand{\fsum}[2]{\sum \limits_{#1}^{#2}}

\newcommand{\fcap}[2]{\bigcap \limits_{#1}^{#2}}

\newcommand{\kl}[1]{\left(#1\right)}

\newcommand{\men}[1]{\left\{#1\right\}}

\newcommand{\anz}[1]{\left|#1\right|}

\newcommand{\ubr}[2]{\underset{#1}{\underbrace{#2}}}

\newcommand{\kw}[1]{\frac{1}{#1}}
\newcommand{\inv}{^{-1}}

\newcommand{\hide}[1]{}

\DeclareMathOperator{\Aut}{Aut}

\DeclareMathOperator{\Char}{char}
\let\NG\relax
\DeclareMathOperator{\NG}{NG}

\DeclareMathOperator{\ex}{ex}

\newcommand{\ep}{\varepsilon}

\newcommand{\isom}{\simeq}

\newtheorem{lemma}{Lemma}
\newtheorem{corollary}{Corollary}

\newtheorem{theorem}{Theorem}
\newtheorem{conjecture}{Conjecture}
\newtheorem{claim}{Claim}

\title{Exploring Projective Norm Graphs}
\author{Tomas Bayer\footnote{Freie Universit\"at Berlin, Institut f\"ur Mathematik (Arnimallee 7, DE-14195 Berlin), tomas.bayer@fu-berlin.de } \and Tam\'as
  M\'esz\'aros\footnote{Freie Universit\"at Berlin, Institut f\"ur Mathematik, (Arnimallee 7, DE-14195 Berlin), tamas.meszaros@fu-berlin.de, Supported by a DRS
    Fellowship of Freie Universit\"at Berlin} \and Lajos
  R\'onyai\footnote{Institute for Computer Science and Control,
    Hungarian Academy of Sciences (L\'agym\'anyosi u. 11, HU-1111 Budapest) and Budapest University of Technology and Economics (Egry J\'ozsef u. 1, HU-1111 Budapest), ronyai.lajos@sztaki.mta.hu} \and Tibor Szab\'o\footnote{Freie Universit\"at Berlin, Institut f\"ur Mathematik (Arnimallee 7, DE-14195 Berlin), szabo@zedat.fu-berlin.de, Supported by GIF grant G-1347-304.6/2016.}}

\begin{document}

\maketitle

\begin{abstract} 
The projective norm graphs $\NG(q,t)$ provide tight constructions for
the Tur\'an number of complete bipartite graphs $K_{t,s}$ with
$s>(t-1)!$. In this paper we determine their automorphism group and
explore their small subgraphs. To this end we
give quite precise estimates on the number of solutions of certain equation
systems involving norms over finite fields.
The determination of the largest integer $s_t$, such that the
projective norm graph $\NG(q,t)$ contains $K_{t,s_t}$ for all large
enough prime powers $q$ is an important open question with far-reaching general
consequences. The best known bounds, $t-1\leq s_t \leq (t-1)!$, are far
apart for $t\geq 4$.
Here we prove that $\NG(q,4)$ does contain (many) $K_{4,6}$ for any
prime power $q$ not divisble by $2$ or $3$. This greatly extends recent work
of Grosu, using a completely different approach.  
Along the way we also count the copies of any fixed $3$-degenerate subgraph, and find that
projective norm graphs are quasirandom with respect to this parameter.
Some of these results also extend the work of Alon and Shikhelman on generalized Tur\'an numbers.
Finally we also give a new, more elementary proof for the $K_{4,7}$-freeness of $\NG(q,4)$.
\vspace{0.2cm}

\end{abstract}

\section{Introduction}

Among both the earliest and most thoroughly studied problems in
extremal graph theory are \emph{Tur\'an}-type problems. 
Given a graph $H$ and integer $n\inn$, the {\em Tur\'an number of $H$}, denoted by $\ex(n,H)$, is
the maximum number of edges 
a simple graph on $n$ vertices may have without containing a subgraph
isomorphic to $H$.
%
%
The very first result about Tur\'an numbers is Mantel's theorem~\cite{Man} from 1907,
stating that $\ex(n,K_3)=\lfloor\frac{n^2}{4}\rfloor$. In 1941 Tur\'an~\cite{Tu}
determined $\ex(n,K_t)$ exactly for every $t\geq 3$ and identified the
unique extremal examples.  For arbitrary $H$, a corollary of the Erd\H{o}s-Stone
Theorem~\cite{ESt}, formulated by Erd\H{o}s and Simonovits~\cite{ESi},
gives  
\begin{equation*}\label{eq:ESS}
\ex(n,H) = \kl{ 1-\frac{1}{\chi(H)-1}} \binom n2+o(n^2),
\end{equation*}
where $\chi(H)$ is the chromatic number of $H$. 
If $H$ is not bipartite, this theorem determines $\ex(n,H)$
asymptotically.

For bipartite graphs $H$ the
Erd\H{o}s-Stone-Simonovits theorem merely states that $\ex(n,H)$ 
is of lower than quadratic order. A general
classification of the order of magnitude of bipartite Tur\'an numbers
is widely open, even in the simplest-looking cases of even cycles and complete bipartite graphs.
Among even cycles the order of magnitude of the Tur\'an number is known
only for $C_4, C_6$ and $C_{10}$~\cite{E,Be}. For the
Tur\'an number of complete bipartite graphs 
a general upper bound,
\begin{equation*}
\ex(n, K_{t,s}) \leq \kw 2 \sqrt[t]{s-1}\cdot n^{2-\kw t}+\frac{t-1}{2}\cdot n,
\end{equation*}
 was proved by K\H{o}v\'ari, T.~S\'os and
Tur\'an~\cite{KST} using an elementary double counting argument.
In general it is commonly conjectured (see e.g. \cite{Bo,CG1}) that the order of magnitude in the
K\H{o}v\'ari-T.S\'os-Tur\'an theorem is the right one. 

\begin{conjecture}\label{conj:KST}
	For every $t,s\inn$, $t\leq s$,
		\[ \ex(n, K_{t,s})  = \Theta\kl{n^{2-\kw t}}.\]
\end{conjecture}

To prove a matching lower bound, one
needs to exhibit a $K_{t,s}$-free graph that is dense enough. 
A general lower bound of $\Omega(n^{2-\frac{s+t-2}{st-1}})$ can be obtained
using the probabilistic method, but this is of smaller order for all
values of the parameters. Constructions with number of
edges matching the order of the upper bound were first found for
$K_{2,2}$-free graphs (attributed to Esther Klein by Erd\H{o}s~\cite{E})  
and later for $K_{3,3}$-free graphs (Brown~\cite{Br}). 
In both cases further analysis~\cite{F1, F2} has also led to the determination of
the correct leading coefficient.

Koll\'ar, R\'onyai and Szab\'o~\cite{KRSz} proved Conjecture
\ref{conj:KST} for every $t\geq 4$ and $s>t!$ by constructing for every $t\inn^+$ a family
of graphs, called norm graphs, that are $K_{t,t!+1}$-free and their
density matches the order of magnitude of the K\"ov\'ari-S\'os-Tur\'an
upper bound. 
Later Alon, R\'onyai and Szab\'o \cite{ARSz} modified this
construction to verify the conjecture for $s>(t-1)!$.
One way or another all these $K_{t,s}$-free constructions of optimal
density are based on the 
simple geometric intuition that $t$ ``average'', ``generic''
hypersurfaces in the $t$-dimensional space are ``expected'' to have a
$0$-dimensional intersection. In manifestations of this idea 
the neighborhoods of vertices are such hypersurfaces and the largest
common neighborhood more or less corresponds to the degree of the intersection.
Recently Blagojevi\'c, Bukh and Karasev~\cite{BBK} and later Bukh~\cite{Bu} implemented the idea in a random setting,
where the neighborhoods are determined by random polynomials. This gave an
alternative proof of the tightness of Conjecture~\ref{conj:KST} for
$s=f(t)$, with $f(t)$ much larger than $(t-1)!$.

Despite significant effort by numerous researchers 
in the last sixty years, the fundamental question
about the Tur\'an number of
$K_{t,t}$ is wide open, even in the case of $t=4$. For 
$\ex(n, K_{4,4})$ or even for $\ex(n, K_{4,6})$, it  is not even known
whether they are of larger order than $n^{\frac{5}{3}} = \Theta (\ex(n, K_{3,3}))$.

\subsection{The projective norm graphs}
Let $N:\fqt \rightarrow \fq$ denote the $\fq$-\emph{norm on} $\fqt$,
i.e. $N(A)=A\cdot A^q\cdot A^{q^2}\cdots A^{q^{t-2}}$ for $A\in
\fqt$. 
For a prime power $q=p^k$ and integer $t\geq 2$ Alon et al.~\cite{ARSz}
defined the \emph{projective norm graph} $\NG(q,t)$ as the graph with
vertex set $\mathbb{F}_{q^{t-1}}\times \mathbb{F}_q^{*}$ and vertices
$(A,a)$ and $(B,b)$ being adjacent if $N(A+B)=ab$. 
The projective norm graph $\NG(q,t)$ has $n=n(q,t) :=q^{t-1} \cdot
(q-1)= (1+o(1))q^t$ vertices. 
To count the edges, one can consider an arbitrary 
vertex $(A,a)\in \fqt\times\fqu$ and
determine its degree. First note that if  $(X,x)\in\fqt\times\fqu$ is
a neighbor of $(A,a)$, then $X\neq -A$, as otherwise $0=N(A+X) =ax$, a
contradiction to $a,x \neq 0$. 
For any other choice $X\in\fqt\backslash\{-A\}$ the value of $x$ is
determined uniquely, namely $x = \kw{a}\cdot N(A+X)$, and hence
$(X,x)$ is a neighbour unless it is the same vertex as $(A,a)$. This
happens exactly if $N(2A)=a^2$. Vertices satisfying the latter
equality will be called \emph{loop vertices}. The degree of a
non-loop vertex then is $q^{t-1}-1$, while it is one less for a loop
vertex. 
The number of loop vertices is  $q^{t-1}-1$ if $\Char(\fq)\neq 2$
and zero if $\Char(\fq)=2$, by parts (f) and (a) of
Lemma~\ref{lemma:NormProperties} of the Appendix, respectively.
Now, the number of edges can be precisely calculated:
\begin{align*}
e(\NG(q,t)) &=  \absw{\kw 2 \kl{q^{t-1} - 1}q^{t-1} (q-1) &\text{if $q$ is a power of $2$} \\ \kw 2 \kl{q^{t-1} - 1}(q^{t-1} (q-1) - 1) &\text{otherwise}}.
\end{align*}
In other words, the number of edges in both cases is $\approx \frac{1}{2}q^{2t-1}
\approx \frac{1}{2}n^{2-\frac{1}{t}}.$
Using a general algebro-geometric lemma from \cite{KRSz}, it was shown
in \cite{ARSz} that $\NG(q,t)$ is $K_{t,(t-1)!+1}$-free. 
Since $\NG(q,t)$ also has the desired density, it verifies Conjecture
\ref{conj:KST} for $s>(t-1)!$. 

Since their first appearance, projective norm graphs were studied
extensively \cite{ASh,BP1,BP2,G,KMV,MW, PTTW}.
Their various properties were utilized in many other
areas, both within and outside combinatorics. These include, among others, (hypergraph) Ramsey theory \cite{AR,KPR,LM,MS,Mu1,Mu2,WLi,WLin}, (hypergraph) Tur\'an
theory \cite{ASh,AKS2,Nik1,Nik2,PTTW,PT}, other problems in extremal combinatorics, \cite{AKS1,BS,Ma,PR,SVe,SVo}, number theory \cite{Nic,RSV,Ve,Vi}, geometry \cite{FPSSZ,PST} and computer science \cite{AMY,BGW,BGKRSzW,DKL}. 

A drawback of the proof of the $K_{t,(t-1)!+1}$-freeness of $\NG(q,t)$
in \cite{ARSz} is that it does not give any information about complete
bipartite subgraphs with any other parameters. 
In particular, not only it is not known for any
$t\geq 4$ whether $\NG(q,t)$ contains a $K_{t,t}$, but it is also not
known whether it contains a $K_{t, (t-1)!}$.
Considering the fundamental nature of Conjecture~\ref{conj:KST}, it
was already suggested in \cite{KRSz} that the determination of the
largest integer $s_t$, such that $\NG(q,t)$ contains $K_{t, s_t}$ for
{\em every} large enough prime power $q$ is a question of great interest.
It is known that $s_2=1$, $s_3=2$, but the bounds for $t\geq 4$ are
very far apart: $t-1\leq s_t \leq (t-1)!$.
If $s_t$ were found to be less than $(t-1)!$ then the
projective norm graphs verified Conjecture~\ref{conj:KST} for 
more values of the parameters than what is known currently. 
In particular, as already mentioned before, for the Tur\'an number of $K_{4,6}$ no better lower
bound than $\ex(n,{K_{3,3}})=\Theta (n^{\frac{5}{3}})$ is known. 

There was/is a reasonable amount of hope that the method of
\cite{ARSz} was not optimal for $\NG (q,t)$, and that the projective
norm graphs might also 
not contain $K_{t,s}$ for some $s \leq (t-1)!$.  This optimism is mainly inspired by
the generality of the key lemma of \cite{KRSz} used in the proof. 
That lemma provides very general conditions, under which the system
of equations  
$$\prod_{j=1}^t (x_j -a_{i,j}) = b_i, \ \ i=1, 2, \ldots , t,$$
over any field $\mathbb{F}$ has at most $t!$ solutions $(x_1, \ldots ,
x_t) \in \mathbb{F}^t$. Namely, it was enough to assume for the
constants $a_{i,j}, b_i \in \mathbb{F}$, that $a_{i_1, j} \neq a_{i_2, j}$ whenever
$i_1\neq i_2$. 
For the application one has to use the lemma for the field
$\mathbb{F}_{q^{t-1}}$ only in the special case when $a_{i,j} =
a_{i,1}^{q^{j-1}}$ for every $i,j\in [t]$, and one is interested in bounding the number
of only those solutions for which $x_j = x_1^{q^{j-1}}$ for every $j=1, \ldots ,
  t$.
That is, the key lemma is used for a very special choice of constants
and very special type of solutions.

In this direction, Ball and Pepe~\cite{BP1,BP2} recently proved that the
$K_{t,(t-1)!+1}$-free projective norm graphs do not contain a
$K_{t+1,(t-1)!-1}$, which in particular improved the earlier
probabilistic lower bound on $\ex(n,{K_{5,5}})$. 

Recently Grosu \cite{G} showed that there is a sequence of primes of
density $\frac{1}{9}$, such that for any prime $p$ in this sequence
$\NG(p, 4)$ does contain a $K_{4, 6}$. 

In this paper we extend this to any prime power $q=p^k$, $p\neq 2, 3$,
and also show the existence of not only one, but many copies of $K_{4,6}$. 
Our method is entirely different from Grosu's.
On the way, we are able to determine asymptotically the number of any
$3$-degenerate subgraphs. This has implications to the quantitative
quasirandom properties of projective norm graphs and extends results of
Alon and Shikhelman~\cite{ASh} on generalized Tur\'an numbers.

Furthermore we also give a new, commutative algebra-free proof of the 
$K_{4,7}$-freeness of $\NG(q, 4)$. This argument extends to estimating the
size of the common neighborhoods of four element vertex sets
in $\NG(q, t)$, for any $t\geq 4$. For $t\geq 5$ this was not known to follow from the commutative 
algebraic proof of \cite{KRSz, ARSz}.

Finally, we are also able to determine the automorphism groups
of $\NG(q, t)$ for every value of the parameters.
In the next four subsections we state our main theorems.

\subsection{Common neighborhoods}

In our first main result we consider the common neighborhood of 
small vertex sets in the projective norm graphs.

Recall that for some vertices $(A, a)$ of $\NG(q,t)$ we might have 
$N(2A)=a^2$, in which case there is a loop edge at $(A, a)$. While for
the graph theoretical applications one is mostly interested in the
simple graph created by deleting these loops, for the purposes of our
statements and proofs, for the rest of this paper it will be convenient to consider 
$\NG(q,t)$ as a graph {\em together} with the loops. Whenever we would still
like to make statements involving simple graphs, we will emphasize this.

For a graph $G$ (with loops) and a set of vertices $T\subseteq V(G)$ we define the
\emph{common neighbourhood} of $T$ as $\mathcal{N}(T)= \fcap{v\in
  T}{}\mathcal{N}(v)$, where $\mathcal{N}(v)$ denotes the set of
neighbours of vertex $v$. The \emph{common degree} of $T$ is  
$\deg(T)= \anz{\mathcal{N}(T)}$.
With this notation the main result of Alon, R\'onyai and
Szab\'o~\cite{ARSz} can be phrased as $\deg(T) \leq (t-1)!$ for every
subset $T\subseteq V(\NG(q,t))$ of size $t$. 

In this direction we investigate the common neighbourhood of pairs,
triples and quadruples of vertices in $\NG(q,t)$. 
A moment of thought reveals that two vertices with the same first
coordinate do not have a common neighbour in $\NG(q,t)$. 
We call a set of vertices in $\NG(q,t)$ \emph{generic}, if the first
coordinates of them are pairwise distinct. In particular, the common
neighborhood of non-generic vertex sets is empty.

Equality of the second coordinates will also turn out to affect, even if to a 
much smaller extent, the size of 
common neighborhoods. To this end we call a set of vertices {\em
  aligned} if {\em all} its elements have the same second coordinate. 
For the precise statement it will be convenient to introduce the
indicator function of a vertex set being aligned. For $T\subseteq V$ let
\begin{equation*}
\xi (T) = \absw{1 & \text{if $T$ is aligned} \\
  0  & \text{if $T$ is not aligned}}.
\end{equation*}
Furthermore let $\eta_{\fq}$ be the quadratic character of $\fq$ if $q$ is odd.
Our results about generic vertex sets are summarized in the following
theorem. 

\begin{theorem}\label{thm:main}
Let $q=p^k$ be a prime power, $t\geq 2$ an integer, and $T= \{ (A_i,
a_i) : i =1, \ldots , j\}$ a generic $j$-subset of vertices in $\NG(q, t)$.
\begin{enumerate}[(a)]
\item If $|T|=2$, then 
\begin{equation*}
\deg(T) = \frac{q^{t-1}-1}{q-1} -\xi(T).
\end{equation*}
In particular, we have $\deg(T)= (1+o(1)) q^{t-2}$.
\item If $|T|=3$ and $q$ is odd, then 
\begin{equation*}
\deg(T) = \absw{1 - \eta_{\fq}\kl{(1+c_1-c_2)^2-4c_1} -\xi(T) & \text{if } t=3, \\ 2q+1 - \eta_{\fq}(-3) -\xi(T) & \text{if } t=4 \text{ and } (c_1,c_2) = (1,-1), \\ q^{t-3} + O(q^{t-3.5})  & \text{otherwise,}} 
\end{equation*}
where $c_1=c_1(T)=\frac{a_1}{a_3}\cdot
N\kl{\frac{A_2-A_3}{A_1-A_2}}\in \mathbb{F} _q$, $c_2=c_2(T)= \frac{a_2}{a_3}\cdot
N\kl{\frac{A_1-A_3}{A_1-A_2}}\in \mathbb{F} _q$.\\
In particular, for $t\geq 4$ we have $\deg(T) = (1+o(1))q^{t-3}$, unless $t=4$ and $(c_1(T),
c_2(T)) = (1, -1)$.
\item If $|T|=4$ and $t\geq 4$ then $\deg(T)\leq 6(q^{t-4}+q^{t-5}+\cdots + q+1)$.
\end{enumerate}
\end{theorem}

One interesting feature of part $(c)$ is that its proof provides a
new argument for the $K_{4,7}$-freeness of $\NG(q,4)$, which uses more
elementary tools than the ones in \cite{KRSz, ARSz}.

Note that for $t=3$ the proof of F\"uredi~\cite{F2}
strengthening the K\"ov\'ari-S\'os-Tur\'an upper bound, coupled with 
the fact that $\NG(q,3)$ is $K_{3,3}$-free implies that roughly half of
the triples in $\NG(q,3)$ must have two common neighbors and roughly half
of them have none. 
In the first case of part $(b)$ we characterize triples of each type.

The information provided in part $(b)$ about the common neighborhood of 
$3$-element sets in $\NG(q,4)$ will enable us to construct a large
number of copies of $K_{4,6}$.

\begin{theorem} \label{thm:K46}
Let $q=p^k$ be any prime power such that $p\neq 2, 3$. 
In the projective norm graph $\NG(q, 4)$ there exists at least 
$(q^3-1)(q-1)(q^3-24)=(1+o(1))q^7$ copies of $K_{4,6}$.
\end{theorem}


\subsection{Quasirandomness}
A (sequence of) graph(s) $G$ on $n$ vertices with average degree $d=d(n)$ is
called {\em quasirandom} if  it possesses some property that the Erd\H os-R\'enyi
binomial random graph $G\left(n,\frac{d(n)}{n} \right)$ also has with probability tending
to $1$ as $n$ tends to infinity. For dense graphs $G$, i.e. when
$\frac{d}{n}$ is constant, many of these natural properties are known
to be equivalent. (see the seminal papers of
Thomason~\cite{Th1, Th2}, and
Chung, Graham, and Wilson~\cite{CGW}). 
These include properties
\begin{itemize}
\item[{\bf Q1}] for any two large enough subsets $A$ and
$B$ of vertices, the number of edges going between them is $\approx \frac{d}{n}|A||B|$;
\item[{\bf Q2}] for most pairs of vertices their common degree is $\approx \frac{d^2}{n}$; 
\item[{\bf Q3}] for any fixed graph $H$, the number of labeled copies
  of $H$ is $\approx n^{v(H)}\left(\frac{d}{n}\right)^{e(H)}$;
\item[{\bf Q4}] the second largest
among the absolute values of eigenvalues of $G$, denoted by $\lambda(
G)$, is of smaller order than the degree $d$ (which is the largest
eigenvalue).
\end{itemize}
The relationship between these properties was investigated in several
papers \cite{CG2, CFZ, KRS} for sparse graphs, i.e., when
$d=o(n)$. Properties {\bf Q1} and {\bf Q2} for example always
follow from {\bf Q4} due to the Expander Mixing Lemma~\cite{ASp}, with a smaller second eigenvalue implying stronger quasirandomness.
Some of the implications however, in contrast to the dense case, are
far from being true. It is an interesting general problem to quantify the extent to which one of these properties implies
another.  

The projective norm graphs in particular serve as examples for some
of the equivalences being false. 
Alon and R\"odl~\cite{AR} and Szab\'o~\cite{Sz} showed that the
eigenvalues of $\NG(q,t)$ are exactly $\pm 1$ times the absolute values of the different Gaussian sums over the
field $\mathbb{F}_{q^{t-1}}$ and hence the second largest absolute
value of an eigenvalue is $\lambda = \lambda(\NG(q,t)) = q^{\frac{t-1}{2}}$. 
That is, not only $\lambda$ is of smaller order
than the degree $d\approx q^{t-1}$, so {\bf Q4} holds, but
$\lambda$ is roughly the square root of the degree.
As it is well-known (and not hard to see,
e.g., \cite{KS}) that for every $d$-regular
graph on $n$ vertices $\lambda = \Omega (\sqrt{d})$ (provided 
the density $\frac{d}{n}$ is bounded away from $1$), 
the projective norm graphs are {\em as quasirandom as it gets}, at
least in terms of their second eigenvalue. 
Still, {\bf Q3} can fail for an arbitrary large inverse polynomial density $n^{-\alpha}$,
$\alpha >0$. 
For example, $\NG(q, 4)$ does not contain any $K_{4,7}$, but 
the random graph $G(n, n^{-\frac{1}{4}})$ contains many (i.e. $\Theta(n^{4})$) copies. 

Even though {\bf Q3} might fail for certain graphs, it is an
interesting problem to quantify to what extent the ``perfect
quasirandomness'' of $\NG(q,t)$ in terms of property {\bf Q4}
carries over to property {\bf Q3}. To this end we will call a graph
$G$ {\em $H$-quasirandom} if property {\bf Q3} holds for $H$, that is,
if the number $X_H(G)$ of labeled copies of $H$ in $G$ is
$\Theta(n^{v(H)}\left(\frac{d}{n}\right)^{e(H)})$. If $X_H(G) =
(1+o(1)) n^{v(H)}\left(\frac{d}{n}\right)^{e(H)}$, then we say that
$G$ is {\em asymptotically $H$-quasirandom}.
With this notion any regular graph is asymptotically $K_2$-quasirandom and the
projective norm graph $\NG(q,t)$ is not $K_{t,(t-1)!+1}$-quasirandom. 
 
Alon and Pudlak~\cite{AP} (see also \cite{KS}) have shown using the Expander 
Mixing Lemma that any $d$-regular graph $G$ on $n$ vertices with
second eigenvalue $\lambda$ (such graphs are also called {\em $(n,d,\lambda)$-graphs}) 
and $\lambda \ll \frac{d^{\Delta}}{n^{\Delta -1}}$ contains
$(1+o(1))n^{v(H)}\left(\frac{d}{n}\right)^{e(H)}$ labeled copies of
any $H$ with maximum degree at most $\Delta$. 
In our terminology they have shown that an $(n, d, \lambda)$-graph
with small enough $\lambda$ is asymptotically $H$-quasirandom.

For the projective norm graph this means that if $\Delta (H) <
\frac{t+1}{2}$, then $\NG(q,t)$ is $H$-quasirandom.
For $\Delta=2$ this statement starts to work when $t$ is at least $4$ and for
$\Delta=3$ it starts to work when $t$ is at least $6$. 
Our second main result goes {\em beyond} what is possible in terms of
subgraph containment from the
general eigenvalue bound of the Expander Mixing Lemma and also 
deals with the much wider class of degenerate graphs instead of just
bounded maximum degree.
(Recall that a graph $G$ is $r$-degenerate if every nonempty subgraph
of $G$ has a vertex of degree at most $r$, or equivalently, there is
an ordering of the vertices of $G$ such that every vertex has at most 
$r$ neighbours preceeding it.)

\begin{theorem}\label{thm:subgraph}
Let $q=p^k$ be an odd prime power and $H$ a simple graph. 
Then for the number of vertex labeled copies of $H$ in $\NG(q,t)$ we have
that, as $q$ tends to infinity,
\begin{equation}\label{eq:order-count}
X_H(\NG(q,t)) = \Theta \left(q^{tv(H) - e(H)}\right),
\end{equation}
provided $H$ is $3$-degenerate and $t\geq 4$. That is, $\NG(q,t)$ is $H$-quasirandom.\\
Moreover, if $H$ is $3$-degenerate and $t\geq 5$ or  $H$ is $2$-degenerate
and $t\geq 3$, then 
\begin{equation}\label{eq:asymptotic-count}
|X_H(\NG(q,t)) - q^{t\cdot v(H) - e(H)}| \leq   O(q^{t v(H)-e(H)-\frac{1}{2}}).
\end{equation}
In particular, $\NG(q,t)$ is asymptotically $H$-quasirandom in these cases. 
\end{theorem}

{\bf Remarks.} 
\begin{itemize}
\item[{\bf 1.}] As $\NG(q,3)$ does not contain $K_{3,3}$ and $\NG(q,2)$
  does not contain $K_{2,2}$, the bound on $t$ for
  \eqref{eq:order-count} is best possible for both $3$- and
  $2$-degenerate graphs.
We conjecture though that the stronger statement \eqref{eq:asymptotic-count} should also be
true for $3$-degenerate graphs and $t=4$.
\item[{\bf 2.}] The theorem remains valid even if $H=H_q$ and $v=v(H_q)$
grows moderately, namely if $v(H_q))=o(\sqrt{q})$ as $q$ tends to infinity, with an error term $o\left(q^{t v(H)-e(H)}\right)$ in (\ref{eq:asymptotic-count}).
\end{itemize}


\subsection{Generalized Tur\'an numbers}

For simple graphs $T$ and $H$ (with no isolated vertices) and a positive
integer $n$ the generalized Tur\'an problem asks for the maximum
possible number $\ex(n,T,H)$ of unlabeled copies of $T$ in an $H$-free graph on
$n$ vertices. Note that by setting $T=K_2$ we recover the original
Tur\'an problem for $H$. Alon and Shikhelman~\cite{ASh} investigated the
problem in the case when $H$ is a complete bipartite
graph $K_{t,s}$ with $t\leq s$, and $T$ is a complete graph $K_\ell$ or a complete
bipartite graph $K_{a,b}$. They have shown that $K_{t,s}$-freeness in
an $n$ vertex graph implies that the number of copies of $T$ is at most
$O\left(n^{v(T)-\frac{e(T)}{t}}\right)$, whenever $T$ is a clique $K_m$
with $m\leq t+1$ or a complete bipartite graph $K_{a,b}$ with $a\leq b
<s$ and $a\leq t$. This, together with the Alon-Pudlak counting of subgraphs in the
projective norm graph implied that for every $s>(t-1)!$, the generalized
Tur\'an number
\begin{equation}\label{eq:general-turan}
ex(n, T,K_{t,s}) = \Theta \left( n^{v(T) - \frac{e(T)}{t}} \right), 
\end{equation}
whenever $T$ is a clique $K_m$ with $m\leq \frac{t+2}{2}$ or a complete
bipartite graph $K_{a,b}$ with $a\leq b \leq \frac{t}{2}$. 
For $T=K_3$, Kostochka, Mubayi and Verstra\"ete~\cite{KMV} and Alon and Shikhelman~\cite{ASh} 
counted triangles in the projective norm graphs more directly, which
extended \eqref{eq:general-turan} from $t\geq 4$ to all $t\geq 2$.



Here we  extend the validity of \eqref{eq:general-turan}, beyond
the eigenvalue bound, for $T=K_4$ and complete bipartite graphs with
one part of size at most three. As each unlabeled copy of $T$ gives rise to
$|\Aut(T)|$ labeled copies, the following is an 
immediate corollary of Theorem~\ref{thm:subgraph}.

\begin{corollary}\label{cor:general-turan}
For every $3$-degenerate simple graph $T$ and any fixed
$t\geq 4$ and $s> (t-1)!$ we have
\begin{equation*}
\ex(n,T,K_{t,s})\geq (1+o(1))\frac{1}{|\Aut(T)|}n^{v(T)-\frac{e(T)}{t}}.
\end{equation*}
\end{corollary}


By combining the upper bound of Alon and Shikhelman~\cite{ASh} with the
lower bound of Corollary~\ref{cor:general-turan} we determine 
the order of magnitude of many new generalized Tur\'an numbers.

\begin{corollary}\label{countingcor2} For every $t\geq 4$ and $s>(t-1)!$ we have 
\begin{equation*} 
ex(n, T, K_{t,s}) = \Theta\left(n^{v(T)-\frac{e(T)}{t}}\right),
\end{equation*}
whenever $T$ is a clique $K_4$ or a complete bipartite graph $K_{a,b}$
with $a\leq b < s$ and $a\leq 3$.
\end{corollary}

\subsection{The automorphism group}

Our last main result concerns the automorphisms of $\NG(q,t)$. 
In the statement $Z_n$ denotes the cyclic group of order $n$.

\begin{theorem}\label{thm:aut}
For any odd prime power $q=p^k$ and integer $t\geq 2$, the maps of the form 
\begin{equation*}
(X,x)\mapsto ( C^2 \cdot X^{p^i}, \pm N(C) \cdot x^{p^i} )
\end{equation*}
are automorphisms of $\NG(q,t)$ for any choice of  $C\in\fqt^*$, and
$i\in [k(t-1)]$. \\
For any $q=2^k$ and integer $t\geq 2$, the maps of the form 
\begin{equation*}
(X,x)\mapsto ( C^2 \cdot X^{p^i}+A, N(C) \cdot x^{p^i} )
\end{equation*}
are automorphisms of $\NG(q,t)$ for any choice of  $C\in\fqt^*$, $A\in
\fqt$, and $i\in [k(t-1)]$.\\
Moreover, for $q>2$ and $t\geq 2$ these include all automorphisms and
the automorphism group has the following structural description:
\begin{equation*}
\Aut(\NG(q,t))\isom 
\left\{\begin{array}{ll}
Z_{q^{t-1}-1}\rtimes Z_{k(t-1)}& \mbox{if $q$ and $t-1$ are both odd}\\
\left(Z_2\times Z_{\frac{q^{t-1}-1}{2}}\right)\rtimes Z_{k(t-1)}& \mbox{if $q$ is odd and $t-1$ is even}\\
\left((Z_p)^{k(t-1)}\rtimes Z_{q^{t-1}-1}\right)\rtimes Z_{k(t-1)}& \mbox{if $q$ is even} 
\end{array}\right.
\end{equation*}
\end{theorem}

\vspace{0.2cm}

Note that if $q=2$ then $\NG(2,t)$ is a complete graph on $2^{t-1}$ vertices, and so $\Aut(\NG(2,t))$ is the whole symmetric group of order $2^{t-1}$.


\vspace{0.5cm}

\paragraph{Organization of the paper.}
This paper is organized as follows. In
Section~\ref{sec:common-neighborhoods} we prove parts $(a)$ and $(b)$
of Theorem~\ref{thm:main} and Theorem~\ref{thm:K46}, where the first
two play an integral part in the third. By using a completely different
machinery, we prove part $(c)$ of Theorem~\ref{thm:main} 
in Section~\ref{sec:quadruples}. 
In Section~\ref{sec:applications} we apply part $(a)$ and $(b)$ of
Theorem~\ref{thm:main} to show Theorem~\ref{thm:subgraph}.
Theorem~\ref{thm:aut} will be handled in
Section~\ref{sec:auto} using an algebraic theorem of Lenstra.
In Section~\ref{sec:conclusion} we conclude with some remarks
and propose a few intriguing open questions. Finally, in the Appendix
we present some of the technical calculations and collect a few
standard algebraic facts for the convenience of the reader.

\section{Common Neighborhoods}\label{sec:common-neighborhoods}

In this section we first prove a useful lemma reformulating
the common degree of an arbitrary vertex set, which will be used
throughout the paper.

Let $\ell \geq 2$ be an integer. For a generic vertex set $U=\men{(A_i,a_i)\mid i\in [\ell]}\subseteq
V(\NG(q,t))$ of size $\ell$ and every $i\in [\ell -1]$ we define elements
$$B_i = B_i(U) := \kw{A_i-A_\ell} \in \mathbb{F}_{q^{t-1}}^* \ \ \
\ \text{and } \ \ \ b_i = b_i(U):=
\frac{a_i}{a_\ell} \cdot N(B_i) \in \mathbb{F}_q^*.$$
Note that as $U$ is a generic set, the $B_i$s are indeed well-defined.
Furthermore as they are non-zero, the $b_i$s are not zero either. 
The equation system 
\begin{equation}\label{neighboureqB}
N\kl{Y + B_i} = b_i \quad \forall i\in
  [\ell -1]
\end{equation}
is strongly related to the one definig the common neighbourhood of $U$.
Let $H(U)$ be the solution set of (\ref{neighboureqB}), i.e.
\begin{equation*} 
H(U) := \men{ Y\in\fqt \mid N\kl{Y + B_i} = b_i, \quad i\in
  [\ell -1] }.
\end{equation*}
\begin{lemma} \label{lemma:DegreeHelper}Let $\ell\geq 2$ be an integer. For any generic vertex set $U=\men{(A_i,a_i)\mid i\in [\ell]}\subseteq
V(\NG(q,t))$ of size $\ell$ we have 
\begin{equation*} 
\deg(U)=|H(U)\setminus \{0\}|.
\end{equation*}
In particular 
\begin{equation*}
\deg(U) = \absw{ \anz{H(U)} - 1 &\text{if $U$ is aligned,}\\
\anz{H(U)} &\text{otherwise.}}
\end{equation*}
\end{lemma}
\begin{proof}
By definition
\begin{equation*}
\mathcal{N}(U)  = \men{ (X,x) \in \fqt \times\fqu \mid \forall i\in [\ell]: N(X+A_i)=a_i\cdot x }.
\end{equation*}
For $(X,x)\in \mathcal{N}(U)$ define $\displaystyle
\phi((X,x))=\frac{1}{X+A_\ell}$.
Note that $\phi((X,x))$ is well-defined, since  
$N(X+A_\ell) = a_\ell \cdot x \neq 0$, hence $X+A_\ell \neq 0$ as well.
We show that $\phi$ is a bijection
between $\mathcal{N}(U)$ and $H(U)\setminus \{ 0\}$.

Clearly $ \phi((X,x)) \neq 0$. Now for  $i\in [\ell-1]$ we have
\begin{equation*}
N(\phi((X,x))+B_i)=N\kl{\kw{A_\ell +X}+B_i} 
%
= N\kl{\frac{A_i+X}{A_\ell +X}} \cdot N(B_i) = \frac{a_i}{a_\ell}  \cdot N(B_i)=b_i,
\end{equation*}
hence, $\phi((X,x))\in H\ohne{0}$.  

As $\mathcal{N}(\mathcal{N}(U))
\supseteq U \neq \emptyset$, the neighborhood
$\mathcal{N}(U)$ is generic, so $\phi$ is injective. 
For the surjectivity of $\phi$ let $Y\in H\backslash\{0\}$. We show
that $\displaystyle \kl{\frac{1}{Y}-A_\ell,\frac{1}{a_{\ell}\cdot N(Y)}}\in
\mathcal{N}(U)$. Note that this is sufficient 
as this vertex is clearly in $\phi^{-1}(Y)$. Indeed for $i\in [\ell-1]$ we have
\begin{equation*}
N\kl{A_i + \kl{\kw{Y}-A_\ell}}= N\kl{\kw{B_i}+\kw{Y}} = N\kl{\frac{Y+B_i}{B_iY}} = \frac{b_i}{N(B_i)N(Y)}= \frac{1}{a_\ell\cdot N(Y)}  \cdot a_i,
\end{equation*}
and for $i=\ell$ we have
\begin{equation*}
N\kl{A_\ell-\kl{\kw{Y}-A_\ell}}= N\kl{\kw{Y}} = \frac{1}{a_\ell\cdot N(Y)}\cdot a_\ell. 
\end{equation*}

For the second statement of the lemma note that $0\in H$ if and only
if $N(B_i)=b_i=\frac{a_i}{a_{\ell}}\cdot N(B_i)$ for every $i\in [\ell
-1]$, which in turn is equivalent to $\ a_1=a_2=\cdots=a_{\ell}$.

\end{proof}

\subsection{Generic pairs}\label{subsec:pairs}


As a simple application of Lemma \ref{lemma:DegreeHelper} we 
prove part $(a)$ of Theorem \ref{thm:main} about the common degree of
generic pairs.

\begin{proof}[Proof of Theorem~\ref{thm:main}(a)]
To use Lemma~\ref{lemma:DegreeHelper} we have to compute $|H(T)|$,
i.e. the number of solutions $Y\in\fqt$ to the equation $N\kl{Y + B}
= b$ where $\displaystyle B = \kw{A_1-A_2}$ and $\displaystyle b =
\frac{a_1}{a_2} \cdot N(B)$. 

By part (f) of \Cref{lemma:NormProperties}, the number of elements $Y$
in the set $N^{-1}(b) -B$ is precisely $\frac{q^{t-1}-1}{q-1}$, so 
part (a) of Theorem~\ref{thm:main}  follows from the second statement of Lemma~\ref{lemma:DegreeHelper}.
\qedhere
\end{proof}

\subsection{Generic triples}\label{triple subsect}

In this subsection we investigate the common neighbourhood of
generic vertex triples and prove part $(b)$ of Theorem \ref{thm:main}.

Let $T=\men{(A_1,a_1),(A_2,a_2),(A_3,a_3)}\subseteq V$ be a
generic triple of vertices in $\NG(q,t)$. Starting from
Lemma~\ref{lemma:DegreeHelper}, 
we give yet another formulation of the common degree $\deg(T)$. 
To this end recall from the statement of Theorem \ref{thm:main} that 
\begin{equation*} 
c_1(T)=\frac{a_1}{a_3}\cdot N\kl{\frac{A_2-A_3}{A_1-A_2}}\in \mathbb{F} _q\ \ \text{
  and } \ \ c_2(T)=
\frac{a_2}{a_3}\cdot N\kl{\frac{A_1-A_3}{A_1-A_2}}\in \mathbb{F} _q,
\end{equation*}
and define the set 
\begin{equation*}
S(T) : = \{ X\in \fqt : N(X) = c_1(T) \mbox{\ and \ }	N(X+1) =
c_2(T) \}.
\end{equation*} 
\begin{lemma}\label{lemma:reduction} For every generic triple
  $T=\men{(A_1,a_1),(A_2,a_2),(A_3,a_3)}\subseteq V$ we have 
\begin{equation*}
\deg(T) = \absw{ \anz{S(T)} - 1 &\text{if } a_1=a_2=a_3  \\ \anz{S(T)} &\text{otherwise}}
\end{equation*}
\end{lemma}

\begin{proof}

Recall that
\begin{equation*}
H(T) = \men{ Y\in\fqt \mid N(Y+B_i) = b_i \quad\text{for}\quad i=1,2}
\end{equation*}
where $\displaystyle B_i = \kw{A_i-A_3}$ and $\displaystyle b_i =
\frac{a_i}{a_3} \cdot N(B_i)$ for $i=1,2$. 

Once we prove $|S(T)| = |H(T)|$,  \Cref{lemma:DegreeHelper}
delivers the statement of Lemma~\ref{lemma:reduction}.
For every $Y\in \fqt$ we define $\phi(Y)=\displaystyle
\frac{Y+B_1}{B_2-B_1}$. Note that as $T$ is generic, we have $B_1\neq
B_2$, so $\phi$ is well defined. Furthermore $\phi$, as a linear
function, is clearly a
bijection from $\fqt$ to $\fqt$. Hence to establish $|H(T)|=|S(T)|$,
all we need to show is that
$Y\in H(T)$ if and only if $\phi(Y) \in S(T)$. This equivalence holds
because each pair of corresponding equations are equivalent.

On the one hand $c_1(T) = \frac{a_1}{a_3}\cdot
N\kl{\frac{A_2-A_3}{A_1-A_2}} = \frac{a_1}{a_3} \cdot
N\kl{\frac{B_1}{B_2-B_1}}$ is equal to $N\kl{\frac{Y + B_1}{B_2-B_1}}
= N(\phi(Y))$ if and only if  $b_1 = N\kl{Y + B_1}$. 

On the other hand $c_2(T) = \frac{a_2}{a_3}\cdot
N\kl{\frac{A_1-A_3}{A_1-A_2}} = \frac{a_2}{a_3} \cdot N\kl{\frac{B_2}{B_2-B_1}}$
is equal to $N\kl{\frac{Y + B_2}{B_2-B_1}} = N\kl{\phi(Y)+1}$ if and
only if $b_2 = N\kl{Y + B_2}$.
\end{proof}
For our analysis we classify generic triples according to the pair
$(c_1(T),c_2(T)) \in \kl{\fqu}^2$. Accordingly for $(c_1, c_2)\in\kl{\fqu}^2$ let
\begin{equation*}
S_t(c_1,c_2)=\{X\in\fqt\ :\ N(X) = c_1 \mbox{ and }N(X+1) = c_2\}.
\end{equation*}
Note that with this notation Lemma \ref{lemma:reduction} just says that if $T\subseteq V$ is a generic triple, then $\deg(T)$ is either $|S_t(c_1(T),c_2(T))|$ or $|S_t(c_1(T),c_2(T))|-1$.

Before coming to the actual proof of part $(b)$ of
Theorem~\ref{thm:main} we still need a crucial preparatory
step. For $(c_1, c_2)\in\kl{\fqu}^2$ and $t\geq 3$ we will define a
polynomial that is strongly related to the defining equations of
$S(T)$, and whose roots we are able to ``locate'' and count.

For this definition we will use norms on two different fields. To distinguish them, for
$t\geq 1$, we put
\begin{equation*}
n_t(Y) = Y^{q^{t-1}+q^{t-2}+\dots+1}\in\fq[Y].
\end{equation*}
For $t\geq 3$ and $(c_1, c_2)\in\kl{\fqu}^2$ we define the polynomial
\begin{equation*}
f_{t,c_1,c_2}(Y)= n_{t-2}(Y+1)\cdot n_{t-2}(Y) + c_1 \cdot n_{t-2}(Y+1) - c_2 \cdot n_{t-2}(Y) \in\fq[Y]. 
\end{equation*}
%
Denote by $R_t(c_1,c_2)$ the set of roots of $f_{t,c_1,c_2}$ in the
algebraic closure $\fqc$ of $\fq$, and by $R^*_t(c_1,c_2)\subseteq
R_t(c_1,c_2)$ the set of multiple roots among them.
In the following lemma we connect the elements of $S_t(c_1, c_2)$ (the
common roots of the equation system $N(X)=c_1, N(X+1)=c_2$) to the 
roots in $R_t(c_1, c_2)$. It turns out that every root of $f_{t,c_1,
c_2}$ is contained in the union of the fields $\mathbb{F}_{q^{t-1}}$ and $\mathbb{F}_{q^{t-2}}$. 
Furthermore all multiple roots are contained in the intersection 
$\fqt \cap \mathbb{F}_{q^{t-2}} = \fq$ and have multiplicity  two.

\begin{lemma}\label{lemma:TripleNeighRec}
For $(c_1,c_2)\in\kl{\fqu}^2$ and $t\geq 3$ we have	

\begin{enumerate}[(i)]
\item $S_t(c_1,c_2) \subseteq R_t(c_1,c_2)$;
\item $S_t(c_1,c_2) \cap \fq = R_t^*(c_1,c_2)$;
\item $\anz{S_t(c_1,c_2)}+\anz{R_t(c_1,c_2)\cap\fqq{t-2}} = 2(q^{t-3}+q^{t-4}+\dots+1)$.
\item For $t\geq 4$ 
\begin{equation*}
\anz{S_t(c_1,c_2)} = 2(q^{t-3}+q^{t-4}+\dots+1) - \fsum{b\in\fqu\ohne{-c_1}}{}\anz{S_{t-1}\kl{b,\frac{bc_2}{b+c_1}}}.
\end{equation*}
\end{enumerate}

\end{lemma}

\begin{proof}
First we prove part $(i)$. Let $X\in S_t(c_1,c_2)$, that is
\begin{equation*}
c_1 = n_{t-1}(X) \mbox{\ and\ }c_2 = n_{t-1}(X+1).
\end{equation*}
Multiplying the equations by $n_{t-2}(X+1)$ and $n_{t-2}(X)$, respectively, and subtracting them from one another we obtain
\begin{equation*}
c_1 n_{t-2}(X+1) - c_2 n_{t-2}(X) = n_{t-1}(X)n_{t-2}(X+1) - n_{t-1}(X+1)n_{t-2}(X).
\end{equation*}
By substituting $n_{t-1}(X)=n_{t-2}(X) X^{q^{t-2}}$ and $n_{t-1}(X+1)=n_{t-2}(X+1) (X+1)^{q^{t-2}}$ we get
\begin{equation*}
c_1 n_{t-2}(X+1) - c_2 n_{t-2}(X) = n_{t-2}(X)n_{t-2}(X+1) X^{q^{t-2}} - n_{t-2}(X+1)n_{t-2}(X) (X+1)^{q^{t-2}}
\end{equation*}
\begin{align*}
&= n_{t-2}(X+1)n_{t-2}(X)\kl{ X^{q^{t-2}} - (X+1)^{q^{t-2}} }
= n_{t-2}(X+1)n_{t-2}(X) (-1).
\end{align*}
This proves that $X$ is a root of $f_{t,c_1,c_2}$, i.e. $X\in R_t(c_1,c_2)$.
		
For part $(ii)$ let us first consider an arbitrary $X\in S_t(c_1,c_2)\cap\fq$. By part
$(i)$ we know that $X$ is a root of $f_{t,c_1,c_2}$. To show that it
is a multiple root, we check that $X$ is also root of the formal
derivative $f'_{t,c_1,c_2}$. 
As $X\notin\men{0,-1}$, the formal derivative $f'_{t,c_1,c_2}$ at $X$ can be expressed as $q^{t-3}+\dots +q+1$ times
\begin{equation*}
 \left( \frac{n_{t-2}(X+1)n_{t-2}(X)}{X} +\frac{n_{t-2}(X+1)n_{t-2}(X)}{X+1}+\frac{c_1n_{t-2}(X+1)}{X+1} - \frac{c_2n_{t-2}(X)}{X}\right).
\end{equation*}
Since $X\in S_t(c_1,c_2)$, we may replace $c_1$ and $c_2$ by
$N(X)=n_{t-2}(X) X^{q^{t-2}}$ and $N(X+1)=n_{t-2}(X+1)
(X+1)^{q^{t-2}}$, respectively. As $q^{t-3}+\dots +q+1=1$ in $\mathbb{F}_q$, this results in
\begin{equation*}
f'_{t,c_1,c_2}(X) = n_{t-2}(X) n_{t-2}(X+1) \kl{ \kw{X} + \kw{X+1} + \frac{X^{q^{t-2}}}{X+1} - \frac{(X+1)^{q^{t-2}}}{X} }.
\end{equation*}
However as $X\in\fq$, we have $X^q=X$, so the last factor simplifies to
\begin{equation*}
\kw{X} + \kw{X+1} + \frac{X}{X+1} - \frac{X+1}{X} = 0,
\end{equation*}
proving that $f'_{t,c_1,c_2}(X)=0$. Consequently $X\in R^*_t(c_1,c_2)$, hence
\begin{equation}\label{lem3iiiold}
S_t(c_1,c_2) \cap \fq \subseteq R_t^*(c_1,c_2).
\end{equation}
Before proving that in (\ref{lem3iiiold}) we have actually equality, we show part $(iii)$.  

We start by bounding the union and
intersection of the sets $S_t(c_1,c_2)$ and $R_t(c_1,c_2)\cap\fqq{t-2}$.
%
By part $(i)$ we have 
\begin{equation*}
S_t(c_1,c_2) \cup \kl{R_t(c_1,c_2)\cap\fqq{t-2}} \subseteq S_t(c_1,c_2) \cup R_t(c_1,c_2)=R_t(c_1,c_2).
\end{equation*}
Since $S_t(c_1, c_2) \subseteq \fqt$ and $\fqt\cap\fqq{t-2}=\fq$, by
$(i)$ and $(\ref{lem3iiiold})$ we obtain
\begin{equation*}
S_t(c_1,c_2) \cap  R_t(c_1,c_2)\cap\fqq{t-2} = S_t(c_1,c_2) \cap \fqq{t-2}=S_t(c_1,c_2) \cap \fq\subseteq R^*_t(c_1,c_2).
\end{equation*}
These two observations together imply
\begin{equation*}
\anz{S_t(c_1,c_2)}+\anz{R_t(c_1,c_2)\cap\fqq{t-2}}\leq \anz{R_t(c_1,c_2)} + \anz{R^*_t(c_1,c_2)}.
\end{equation*}
Now note that as $\anz{R_t(c_1,c_2)}$ is the  number of different
linear factors of $f_{t,c_1,c_2}$ in $\overline{\mathbb{F}}_q$ and $\anz{R^*_t(c_1,c_2)}$ is the number of different linear factors that appear at least twice, their sum is necessarily bounded from above by the degree of $f_{t,c_1,c_2}$ i.e. by $2(q^{t-3}+\dots+q+1)$. This proves
\begin{equation}\label{lem3iiilower}
\anz{S_t(c_1,c_2)}+\anz{R_t(c_1,c_2)\cap\fqq{t-2}} \leq 2(q^{t-3}+q^{t-4}+\dots+1).
\end{equation}
To get the desired equality for every pair $(c_1, c_2) \in (\fq^*)^2$
we will use a Stepanovesque trick of considering their average and
using double counting to transfer the difficult task of bounding the 
number of solutions of a high degree equation into the easy task of bounding the
number of solutions of a linear equation.
In other words we will show that the desired equality holds for the average, i.e., 
\begin{equation}\label{eq:average}
\kw{(q-1)^2}\fsum{c_1\in\fqu}{}\fsum{c_2\in\fqu}{}\kl{\anz{S_t(c_1,c_2)}+\anz{R_t(c_1,c_2)\cap\fqq{t-2}}} = 2(q^{t-3}+q^{t-4}+\dots+1).
\end{equation}
Note that this indeed will be enough, as we have already obtained the
same upper bound for the individual terms, so equality for the average
is possible only if each individual term matches the upper bound. 
				
To prove \eqref{eq:average}, we split the sum and evaluate each part 
separately. For the first part we use double-counting to obtain
\begin{align*}
\fsum{c_1\in\fqu}{}\fsum{c_2\in\fqu}{}\anz{S_t(c_1,c_2)}
&= \fsum{X\in\fqt}{}\anz{\men{(c_1,c_2)\in\kl{\fqu}^2 \mid
  S_t(c_1,c_2) \ni X}}\\
&= \fsum{X\in\fqt\ohne{0,-1}}{} 1 = q^{t-1}-2.
\end{align*}
The next to last equality holds since the sets $S_t(c_1, c_2)$
partition $\fqt \setminus \{0, -1\}$. Indeed, each $X \in \fqt \setminus \{0, -1\}$ is contained in exactly one of
them, namely $S_t(N(X), N(X+1))$.
Similarly, 
\begin{equation*}
\fsum{c_1\in\fqu}{}\fsum{c_2\in\fqu}{}\anz{R_t(c_1,c_2)\cap\fqq{t-2}}= \fsum{X\in\fqq{t-2}}{}\fsum{c_1\in\fqu}{}\anz{\men{c_2\in\fqu \mid X \in R_t(c_1,c_2) }}.
\end{equation*}
Now for fixed $X\in \fqq{t-2}$ and $c_1\in \fqu$ the expression
$f_{t,c_1,c_2}(X)$ becomes a linear polynomial in $c_2$. It has no
root in $\fq^*$ if $X\in\men{0,-1}$ or $c_1=-n_{t-2}(X)$, otherwise there is a unique $c_2$ for which $f_{t,c_1,c_2}(X)=0$, namely $\displaystyle c_2=\frac{n_{t-2}(X+1)(n_{t-2}(X) + c_1)}{n_{t-2}(X)}$. Hence
\begin{equation*}
\fsum{c_1\in\fqu}{}\fsum{c_2\in\fqu}{}\anz{R_t(c_1,c_2)\cap\fqq{t-2}}=\fsum{X\in\fqq{t-2}\ohne{0,-1}}{}\fsum{c_1\in\fqu\ohne{-n_{t-2}(X)}}{}1= (q^{t-2}-2)(q-2).
\end{equation*}
Summing up both parts, we get
\begin{align*}
\fsum{c_1\in\fqu}{}\fsum{c_2\in\fqu}{}\kl{\anz{S_t(c_1,c_2)}+\anz{R_t(c_1,c_2)\cap\fqq{t-2}}}
  &  =\kl{(q^{t-1}-2)+(q^{t-2}-2)(q-2)} \\
& = 2(q-1)^2(q^{t-3}+\cdots+1),
\end{align*}
which proves \eqref{eq:average}.

Now we turn back to finish the proof of $(ii)$. The equality in $(iii)$ implies that in the proof of (\ref{lem3iiilower}) all displayed inequalities and containments must hold with equality, in particular, we have equality in $(\ref{lem3iiiold})$ as well. 
        
Finally, we prove $(iv)$. To express $|S_t(c_1,c_2)|$ we first count
the elements of $R_t(c_1,c_2)\cap\fqq{t-2}$  through classifying them by
their $(t-2)$-norm and then use part $(iii)$.
\begin{align*}
\anz{R_t(c_1,c_2)\cap\fqq{t-2}} 
&= \fsum{b\in\fq}{}\anz{\men{X\in R_t(c_1,c_2)\cap\fqq{t-2}  : n_{t-2}(X)=b}}\\
&= \fsum{b\in\fq}{}\anz{\men{X\in \fqq{t-2} : n_{t-2}(X)=b \text{ and } n_{t-2}(X+1)(b + c_1) =  b \cdot c_2}}.
\end{align*}
 Note that $0\notin R_t(c_1,c_2)\cap\fqq{t-2}$, since $c_1\neq
 0$. Hence for $b=0$ this set is empty. Moreover it is also empty for $b=-c_1$,
 since neither $c_1$, nor $c_2$ is $0$. Consequently,
\begin{align*}
\anz{R_t(c_1,c_2)\cap\fqq{t-2}} 
&= \fsum{b\in\fqu\ohne{-c_1}}{}\anz{\men{X\in \fqq{t-2} \mid  n_{t-2}(X)=b \text{ and }n_{t-2}(X+1) = \frac{b \cdot c_2}{b+c_1}}}\\
&= \fsum{b\in\fqu\ohne{-c_1}}{}\anz{S_{t-1}\kl{b,\frac{b \cdot c_2}{b+c_1}}}
\end{align*}
Now, the assertion of part $(iv)$ follows by part $(iii)$.
\end{proof}

We are now ready to complete the proof of part $(b)$ of
Theorem~\ref{thm:main}.
\begin{proof}[Proof of Theorem~\ref{thm:main}(b)]
We shall use Lemma~\ref{lemma:reduction}.
We start by examining the case $t=3$. By \Cref{lemma:TripleNeighRec}$(iii)$
\begin{align*}
S_3(c_1,c_2) 
&=2-|R_3(c_1,c_2)\cap\fq|= 2 - \anz{\men{ X\in\fq \mid f_{3,c_1,c_2}(X) = 0 }} \\
&= 2 - \anz{\men{ X\in\fq \mid (X+1)X + c_1 \cdot (X+1) - c_2  X = 0 }} \\
&= 2 - \anz{\men{ X\in\fq \mid X^2 + (1+c_1-c_2)X + c_1 = 0 }} \\
&= 2 - \kl{ 1 + \eta\kl{(1+c_1-c_2)^2-4c_1} } = 1 - \eta\kl{(1+c_1-c_2)^2-4c_1},
\end{align*}
where $\eta=\eta_{\mathbb{F}_q}$ is the  quadratic character of $\mathbb{F}_q$. In the case $t=4$ we apply \Cref{lemma:TripleNeighRec}$(iv)$ and use the case $t=3$ to obtain
\begin{align*}
S_4(c_1,c_2) 
&= 2(q+1)-\fsum{b\in\fqu\ohne{-c_1}}{}\anz{S_{3}\kl{b,\frac{bc_2}{b+c_1}}} \\
&= 2(q+1)-\fsum{b\in\fqu\ohne{-c_1}}{}\kl{1 - \eta\kl{\left(1+b-\frac{bc_2}{b+c_1}\right)^2-4b}} \\
&= q+4+\fsum{b\in\fqu\ohne{-c_1}}{}\eta\kl{\frac{\big((b+c_1)(1+b)-bc_2\big)^2-4b(b+c_1)^2}{(b+c_1)^2}}.
\end{align*}
Put 
\begin{align*}
L(b)
&=\big((b+c_1)(1+b)-bc_2\big)^2-4b(b+c_1)^2\\
&=b^4+2(c_1-c_2-1)b^3+\big((1+c_1-c_2)^2-6c_1\big)b^2+2c_1(1-c_1-c_2)b+c_1^2,
\end{align*}
and observe that the denominator inside $\eta$ may be omitted as it is a non-zero square and $\eta$ is multiplicative. Accordingly, 
\begin{align*}
S_4(c_1,c_2) 
&= q+4+\fsum{b\in\fqu\ohne{-c_1}}{}\eta(L(b)) = q+4-\eta(L(0))-\eta(L(-c_1))+\fsum{b\in\fq}{}\eta(L(b)) \\
&= q+4-\eta(c_1^2)-\eta(c_1^2 c_2^2)+\fsum{b\in\fq}{}\eta(L(b)) = q+2+\fsum{b\in\fq}{}\eta(L(b)).
\end{align*}

Our goal is to use the Weil character sum estimate (see \Cref{thm:Weil} in the Appendix) for the quadratic
character $\eta$. As the order of $\eta$ is $2$, we can estimate the
above sum using the first part of \Cref{thm:Weil} unless
$L(b)=(b^2+\alpha_1b+\alpha_0)^2$ for some $\alpha_1,\alpha_0\in \fq$,
in which case the second part of the same theorem applies. After
expanding and carefully comparing coefficients  (see Claim \ref{coeffclaim} in the Appendix)
one obtains that the latter is possible if and only if
$(c_1,c_2)=(1,-1)$, and in this case we have $L(b) = (b^2+b+1)^2$. 

Accordingly, if $(c_1,c_2)=(1,-1)$, then by the second part of \Cref{thm:Weil}
\begin{align*}
S_4(1,-1)
&= q+2 + \fsum{b\in\fq}{}\eta\big(1\cdot (b^2+b+1)^2\big)= q+2 + \Big(q-\anz{\men{b\in\fq \mid b^2+b+1=0}}\Big)
 \eta(1) \\
&= q+2 +\big (q - \kl{1+\eta(-3)}\big) \cdot 1 = 2q + 1 - \eta(-3).
\end{align*}
Otherwise, if $(c_1,c_2)\neq (1,-1)$, then by the first part of \Cref{thm:Weil} we get
\begin{equation*}
\anz{ S_4(c_1,c_2) - q}= \anz{2 + \fsum{b\in\fq}{}\eta(L(b))} \leq 2 + \anz{\fsum{b\in\fq}{} \eta(L(b))}\leq 2 + (4-1)\sqrt{q}=O(\sqrt{q}),
\end{equation*}
implying that $S_4(c_1,c_2)=q+O(\sqrt{q})$.

For $t=5$  we use \Cref{lemma:TripleNeighRec}$(iv)$ and the case
$t=4$. 
\begin{align*}
S_5(c_1,c_2)
&= 2(q^2+q+1) - \fsum{b\in\fqu\ohne{-c_1}}{}\anz{S_{4}\kl{b,\frac{bc_2}{b+c_1}}} \\
&= 2(q^2+q+1) - S_4\left(1,\frac{c_2}{1+c_1}\right) - \fsum{b\in\fqu\ohne{-c_1, 1}}{}\anz{S_{4}\kl{b,\frac{bc_2}{b+c_1}}} \\
&= 2(q^2+q+1) - O(q) - (q-3) \cdot \kl{ q + O(\sqrt{q}) } =  q^2 + O(q^{1.5})
\end{align*}
Note that in the above estimate it was crucial that we could use that
for most values of $b$, the value of $|S_4 (b,bc_2/(b+c_1))|$ is
asymptotically $q$. 

For $t\geq 6$ we can apply induction with base case $t=5$. The induction step is the same as above, only that now we do not need to distinguish between cases. Indeed suppose that the statement holds for all $5\leq t'<t$ and consider the general case. By  \Cref{lemma:TripleNeighRec}$(iv)$ and the induction hypothesis for $t'=t-1$ we obtain
\begin{align*}
S_t(c_1,c_2)
&= 2(q^{t-3}+\dots+1) - \fsum{b\in\fqu\ohne{-c_1}}{}\anz{S_{t-1}\kl{b,\frac{bc_2}{b+c_1}}}\\
&= 2(q^{t-3}+\dots+1) -  (q-2) \cdot \kl{ q^{t-4} + O(q^{t-4.5}) }=  q^{t-3} + O(q^{t-3.5}).
\end{align*}
This finishes the proof of \Cref{thm:main}(b).
\end{proof}

\subsection{Finding a $K_{4,6}$}

In this section we prove Theorem~\ref{thm:K46}.  That is, assuming
$p\not=2,3$ we will construct (many) quadruples of vertices in $\NG(q, 4)$ which have six
common neighbors. Finally we will see that most of these in fact
involve four plus six {\em different} vertices (no loops involved
among the 24 adjacencies), hence forming a subgraph isomorphic to $K_{4,6}$. 

We will lean heavily on what we have learned about the common neighborhood of triples in the previous
subsection (part $(b)$ of Theorem~\ref{thm:main}), both in terms of intuition and actual tools. 
For $t=4$ we have proved that a small fraction of the triples $T$ had twice as many
common neighbours as the rest and we characterized them as being those
for which $c_1(T)=1$ and $c_2(T)=-1$.
Heuristically one could think that it should be easier to find
quadruples with $6$ neighbours among those which contain such exceptional
triples and maybe even more of them. This is the direction we will be
going and identify those quadruples which contain two such special
triples and prove that indeed roughly half of them have six common neighbors.

We will start to work out the heuristics described above by
investigating this exceptional case and trying to understand better
the algebraic structure of $S_4(1,-1)$. First we observe that the polynomial
$f_{4,1,-1}$ from the proof of the previous subsection can be written
in a product form. 
\begin{align}\label{eq:product-form} 
f_{4,1,-1}(X) & =\kl{X+1}^{q+1}X^{q+1}+(X+1)^{q+1}+X^{q+1}\nonumber\\
&=X^{2q+2}+X^{2q+1}+X^{q+2}+3X^{q+1}+X^q+X+1=h(X,1)\cdot h(1,X),
\end{align}
where 
\begin{equation*} 
h(Y,Z)=Y^{q+1}+Y^qZ+Z^{q+1}.
\end{equation*} 
%
For general $c_1,c_2$ the polynomial $f_{4,c_1,c_2}$ can have roots in
$\fqq{3}$ which are not in $S_4(c_1, c_2)$. In the next lemma we show that
this does not happen when $c_1=1, c_2=-1$, i.e., we ``find'' all the roots in the algebraic closure $\overline{\mathbb{F}}_{q^3}$ of $\fqq3$.

\begin{lemma}\label{lemma:S4bijection_a}
For every prime $p\geq 2$ we have
\begin{align*}
S_4(1,-1) &= \men{X\in\overline{\mathbb{F}}_{q^3}\mid h(X,1)\cdot h(1,X)=0}.
\end{align*}
\end{lemma}
\begin{proof}
 $S_4(1,-1)\subseteq R_4(1,-1)=\men{X\in\overline{\mathbb{F}}_{q^3}\mid h(X,1)\cdot h(1,X)=0}$ by
 \eqref{eq:product-form} and part $(i)$ of \Cref{lemma:TripleNeighRec}.

Now let $X\in\overline{\mathbb{F}}_{q^3}$ be such that $h(X,1)\cdot h(1,X)=0$. Then either
\begin{equation*}
h(X,1)=0 \imp X^q = -\kw{X+1}=u(X)\quad \text{ or }\quad h(1,X)=0 \imp X^q = -\frac{X+1}{X}=v(X).
\end{equation*}
In the first case
\begin{equation*}
X^{q^2} = u(u(X)) = -\kw{-\kw{X+1}+1} = v(X)
\end{equation*}
and
\begin{equation*}
X^{q^3}=u(u(u(X)))=u(v(X))=-\kw{-\frac{X+1}{X}+1}=X,
\end{equation*}
while in the latter case
\begin{equation*}
X^{q^2} = v(v(X)) = -\frac{-\frac{X+1}{X}+1}{-\frac{X+1}{X}} = u(X)
\end{equation*}
and
\begin{equation*}
X^{q^3}=v(v(v(X)))=v(u(X))=-\frac{-\kw{X+1}+1}{-\kw{X+1}}=X.
\end{equation*}
In particular, in both cases we have $X\in \fqq3$ and $X^{q^2+q}=u(X)v(X)$. Accordingly
\begin{equation*}
N(X) = X^{q^2+q+1} = X\cdot u(X)\cdot v(X) = X\cdot \kl{-\kw{X+1}}\cdot\kl{-\frac{X+1}{X} } = 1.
\end{equation*}
Similarly, for the norm of $X+1$ we get
\begin{align*}
N(X+1)&= (X+1)(X^q+1)(X^{q^2}+1)= (X+1)(u(X)+1)(v(X)+1)\\
&= (X+1) \cdot \kl{ -\frac{X+1}X+1}\cdot\kl{-\kw{X+1}+1}= -1.
\end{align*}
This shows that $X\in S_4(1,-1)$ and hence $\men{X\in\overline{\mathbb{F}}_{q^3}\mid h(X,1)\cdot h(1,X)=0}\subseteq S_4(1,-1)$.
\end{proof}

Third roots of unity will play an important role in our further considerations.
Whenever $p\neq 3$, there exists a non-trivial third root of unity
$a\neq 1$ in $\fqc$, that is, a root of the polynomial $Y^2+Y+1$. 
Then $a^2=a^{-1}$ is the other non-trivial root of unity.
Since $a$ and $a^{-1}$ are roots of a quadratic polynomial over
$\fq$, they are both contained in $\fqq{2}$. Let $e_q$ be $-1$, $0$ or $1$ according to whether $q$ is $-1$, $0$ or $1$ modulo 3. Then the polynomial $Y^2+Y+1$ has exactly $1+e_q$ roots in
$\mathbb F_q$. By looking at the order of the multiplicative groups $\mathbb{F}_{q} ^*$
and of $\mathbb{F}_{q^{2}}^*$ we see that if  
$q\equiv -1\mod 3$ then $a, a^{-1}\in \fqq 2\setminus \fq$
and if $q\equiv 1\mod 3$ then $a, a^{-1}\in \fq$. Note that since $a^3=1$, we have
$a^q=a^{e_q}$. We will also make use of the fact that $1+a^{e_q} + a^{-e_q} =0$.

 Let us denote by $G$ and $G^3$ the multiplicative groups of
 $3(q-e_q)$-th and $(q-e_q)$-th roots of unity in the algebraic
 closure $\fqc$ of $\fq$, respectively.
In other words
\begin{equation*}
G = \men{ x\in \fqc \mid x^{3(q-e_q)} = 1}\text{ and }G^3 = \men{ x\in \fqc \mid x^{q-e_q} = 1}.
\end{equation*}

Now suppose that $p\neq 3$, and let us fix a non-trivial third root of
unity $a\in\fqc$ for the rest of this subsection. The linear
fractional transformation $C: \fqc\setminus\{a\} \rightarrow \fqc$, defined by
\begin{equation*}
C(z)=\frac{z - a^{-1}}{z - a},
\end{equation*}
will be instrumental in our arguments.

\begin{lemma}\label{lemma:S4bijection_b}
If $p\neq 3$ then the map $z \mapsto C(z)$ is a bijection from $S_4(1,-1)\setminus\fq$ to $G \setminus G^3$.
\end{lemma}
\begin{proof}
Let $z\in S_4(1,-1)\setminus\fq\subseteq \fqq 3\setminus\fq$. As $a,a^{-1}\in\fqq{2}$ and $\fqq 2\cap\fqq 3=\fq$, we clearly have $z\neq a,a^{-1}$, hence $C(z)$ exists and is nonzero. We aim to show that $C(z)\in G \setminus G^3$, which happens exactly if $C(z)^{q-e_q}\neq 1$ but $C(z)^{3(q-e_q)}=1$.
\begin{equation*}
C(z)^{q} = \frac{(z-a^{-1})^q}{(z-a)^q}=\frac{z^q-a^{-q}}{z^q-a^q}= \frac{z^q-a^{-e_q}}{z^q-a^{e_q}}
\end{equation*}
As $z\in S_4(1,-1)$, by Lemma \ref{lemma:S4bijection_a} we either have $z^q=-\kw{z+1}$ or $z^q = - \frac{z+1}z$. In the first case
\begin{equation*}
C(z)^{q} = \frac{-\kw{z+1}-a^{-e_q}}{-\kw{z+1}-a^{e_q}}=
\frac{za^{-e_q}+1+a^{-e_q}}{za^{e_q}+1+a^{e_q}} = \frac{za^{-e_q}-a^{e_q}}{za^{e_q}-a^{-e_q}}= \frac{z-a^{-e_q}}{z-a^{e_q}} \cdot a^{e_q} = C(z)^{e_q} \cdot a^{e_q}.
\end{equation*}
where we used that $1+a^{e_q}+a^{-e_q}=0$ and $a^{-2e_q}=a^{e_q}$.

In the second case
a similar calculation shows that 
\begin{equation*}
C(z)^{q} = \frac{-\frac{z+1}{z}-a^{-e_q}}{-\frac{z+1}{z}-a^{e_q}}=
\frac{za^{-e_q}+z+1}{za^{e_q}+z+1} = \frac{za^{e_q} - 1}{za^{-e_q} - 1}= \frac{z-a^{-e_q}}{z-a^{e_q}} \cdot a^{-e_q} = C(z)^{e_q} \cdot a^{-e_q}.
\end{equation*}

As $p\neq 3$ we have $e_q\neq 0$, so in both cases $C(z)^{q-e_q}\neq 1$ and hence $C(z)\notin G^3$. On the other hand $C(z)^{3(q-e_q)}=\kl{a^{\pm e_q}}^3=\kl{a^3}^{\pm e_q}=1$, confirming that $C(z)\in G\setminus G^3$.
				
$C$ is injective because it is a nontrivial fractional linear map. 
Then to verify that $C$ is indeed a bijection between $S_4(1,
-1)\setminus \fq$ and $G\setminus G^3$ it is enough to show that the
two sets are of the same size. As $G^3$ is fully contained in $G$ as a
subgroup, the set  $G\setminus G^3$ has $|G|-|G^3|=2(q-e_q)$ elements. 
By the prooof of \Cref{thm:main}$(b)$, the set $S_4(1,-1)$ has $2q+1-e_q$ elements
and by Lemma \ref{lemma:S4bijection_a} we have $S_4(1,-1)\cap
\fq=\men{X\in \fq\mid h(X,1)\cdot h(1,X)=0}$. When viewed as
polynomials over $\fq$, using the identity $Y^q=Y$, both $h(Y,1)$ and
$h(1,Y)$ simplify to the quadratic polynomial $Y^2+Y+1$. As noted
earlier, this polynomial has $1+e_q$ roots over $\fq$ hence
$\anz{S_4(1,-1) \setminus \fq} = 2q+1-e_q - \kl{1+e_q} = 2
\kl{q-e_q}$.  Consequently, 
$C$ is indeed a bijection.
\end{proof}

\begin{lemma}\label{lemma:S4bijection_d}
Let $p\neq 3$. If $A\in S_4(1,-1)\setminus \fq$ then $aA, a^{-1}A \notin S_4(1,-1)$.
\end{lemma}

\begin{proof}
Assume to the contrary that $c A\in S_4(1,-1)\subseteq \fqq 3$ for
$c=a$ or $a^{-1}$. As we have $A\in\fqq{3},$ this
implies that $c\in\fqq{3}$. However $c \in\fqq{2}$, so $c$ also belongs to
$\fqq 3\cap \fqq 2=\fq$. As discussed earlier, a non-trivial third
root of unity is present in $\fq$ if and only if $1=e_q \equiv
q \pmod 3$ and hence $G^3=\fqu$. 

As $A\notin \fq$ we must then also have $cA \notin\fq$, so by Lemma
\ref{lemma:S4bijection_b} both $C(A)$ and $C(cA)$ belong to
$G\setminus G^3$. 
Substituting, we obtain
\begin{align*}
C(aA) &= \frac{aA-a^{-1}}{aA-a} = \frac{A-a^{-2}}{A-1}=\frac{A-a}{A-1}\quad \text{and}\\ C(a^{-1}A) &= \frac{a^{-1}A-a^{-1}}{a^{-1}A-a} = \frac{A-1}{A-a^2}=\frac{A-1}{A-a^{-1}},
\end{align*}
which in particular implies that $C(A)\cdot C(aA)\cdot
C(a^{-1}A)=1$. Since two of the three factors are in $G$, so must be
the third. 

Now we propose two ways to prove that this implies that $A\in \mathbb F_q$ and hence gives a contradiction. 

By the definition of $G$, $C(A)^3$, $C(aA)^3$
and $C(a^{-1}A)^3$ all have to be roots of the polynomial $Y^{q-1}-1=0$ and
hence belong to $\fq$. A straightforward but tedious computation 
shows that $A$ can be expressed as
\begin{equation*}
A=\frac{aC(A)^3C(aA)^3+a^2C(aA)^3+1}{1-a^{2}C(A)^3C(aA)^3-aC(aA)^3}.
\end{equation*}
Since all the ingredients were shown to be in $\fq$, so has to be $A$.

Alternatively, one can observe that none of the elements $C(aA)^3$
and $C(a^{-1}A)^3$ can belong to $G^3=\mathbb F_q^*$, as otherwise this would already imply $A\in \mathbb F_q$. Accordingly by Lemma \ref{lemma:S4bijection_d} the elements $A$, $aA$, $a^{-1}A$ all belong to $S_4(1,-1)\setminus \mathbb F_q$  and hence by Lemma \ref{lemma:S4bijection_a} are roots of either $h(X,1)$ or $h(1,X)$. Let $z_1$ and $z_2$ be two of these three elements which are roots of the same of the two polynomials $h(X,1)$ and $h(1,X)$, and set $B=C(z_1)$, $B'=C(z_2)$. Note that $B$ and $B'$ are both of the form $C(dA)$, where $d$ is a third root of
unity.
Now the formulas for the Galois action on the elements $C(z)$ where $z\in S_4(1,-1)$ (given in the proof of Lemma 5) imply at once
that $(B/B')^q=B/B'$, hence this fraction is in $\mathbb F_q$.
Also $B/B'$ is not 1, as $C$ is injective.
But $B/B'$ is also a fraction $D/D'$, where $D,D'$ are
quadratic polynomials of $A$ with coefficients from $\mathbb F_q$ and with leading coefficient 1.
These imply that $A$ is a root of a nontrivial quadratic equation
over $\mathbb F_q$ and hence $A\in \mathbb{F}_{q^2}$. As $A\in \mathbb F_{q^3}$, this is only possible if $A\in \mathbb F_q$. 
\end{proof}

The next lemma will be the main tool establishing that certain types of equation
systems have six solutions.

\begin{lemma} \label{lemma:SpecialQuadruples} Let $p\neq 2,3$ and $A,B\in\fqq 3$ such that $N(A)=N(B)=1$ and $\frac{A}{B}\in S_4(1,-1)\setminus \fq$. Then the system
\begin{equation}\label{eq0}
N(Y)=1\quad \quad N(Y+A) = -1\quad \quad N(Y+B)=-1
\end{equation}
has $5+\eta_{\fqq 3}\kl{A^2+AB+B^2}$ solutions for $Y$ in $\fqq 3$.
\end{lemma}

\begin{proof}
We divide the first two equations of \eqref{eq0} by $N(A)=1$ to obtain
\begin{equation}\label{eqA}
N\kl{\frac YA} =   1 \quad\quad N\kl{\frac YA + 1} = -1 
\end{equation} 
and the first and the third equations by $N(B)=1$ to get
\begin{equation}\label{eqB}
N\kl{\frac YB} = 1 \quad\quad N\kl{\frac YB + 1} = -1.
\end{equation}
Now clearly, $Y\in \fqq{3}$ is a solution of \eqref{eq0} if and only if it
is a common solution of the systems \eqref{eqA} and \eqref{eqB},
which, by definition just means that both $\frac{Y}{A}, \frac{Y}{B} \in S_4(1,-1)$.
By \Cref{lemma:S4bijection_a} an element $Y\in \fqq{3}$ is a solution 
of system (\ref{eqA}) exactly when it is a solution of
$h\left(\frac{Y}{A}, 1\right) =0$ or  $h\left(1, \frac{Y}{A}\right)
=0$ in $\fqq{3}$, which is in turn is equivalent to being solution of
\begin{equation}\label{eq1}
h(Y,A)=0\quad\text{ or }\quad h(A,Y)=0.
\end{equation}
Analogously, for an element $Y$ of $\fqq{3}$ being a solution of the system (\ref{eqB}) is
equivalent to being a solution of
\begin{equation}\label{eq2}
h(Y,B)=0\quad\text{ or }\quad h(B,Y)=0.
\end{equation}
Consequently, the task of solving (\ref{eq0}) reduces to solving the
four possible combinations of equations from (\ref{eq1}) and (\ref{eq2}):
\begin{alignat}4
h(Y,A)&=0&\quad&\text{ and }\quad h(Y,B)&=0, \label{f1}\\
h(Y,A)&=0&\quad&\text{ and }\quad h(B,Y)&=0, \label{f2} \\ 
h(A,Y)&=0&\quad&\text{ and }\quad h(Y,B)&=0, \label{f3}\\ 
h(A,Y)&=0&\quad&\text{ and }\quad h(B,Y)&=0. \label{f4}
\end{alignat}
Before solving these systems, we show that their solution sets are disjoint, i.e.,
the number of solutions to (\ref{eq0}) is the sum of the number of solutions to (\ref{f1}), (\ref{f2}), (\ref{f3}) and (\ref{f4}).
Indeed, assume that
there was a solution $X\in \fqq{3}$ of \eqref{eq0}, satisfying both
equations in (\ref{eq1}). Then we have
\begin{equation*} 
0=\frac{(X+A)\cdot h(A,X)-X\cdot h(X,A)}{A^{q+2}}=\kl{\frac XA}^2 + \frac XA + 1, 
\end{equation*}
hence, $\frac XA\in \fqq 3$ is a non-trivial third root of unity. Then
by our assumption on $\frac{A}{B}$ and Lemma \ref{lemma:S4bijection_d}
we have that $\frac{X}{B}=\frac{X}{A}\frac{A}{B}\notin S_4(1,-1)$. 
This contradicts the fact that $X$, as a solution to \eqref{eq0}, is
also a solution to \eqref{eqB}. 
By symmetry---note that $\frac{B}{A}\in
S_4(1,-1)\setminus \fq$---we see that there is no solution of  (\ref{eq0}) which solves both
euqations of (\ref{eq2}). 
	
Now we turn to counting the solutions of (\ref{f1}), (\ref{f2}), (\ref{f3}) and (\ref{f4}).	
By expressing $Y^q$ everywhere and setting the respective expression  equal to each other in the systems above, we obtain the following equations.
\begin{alignat}1
0 &= \kl{A^{q+1}-B^{q+1}} \cdot Y + A^{q+1}B-AB^{q+1} \tag{\ref{f1}*}\label{f1*}\\
0 &= B^q\cdot Y^2 + \kl{-A^{q+1}+AB^q+B^{q+1}} \cdot Y + AB^{q+1}\tag{\ref{f2}*}\label{f2*}\\
0 &= A^q\cdot Y^2 + \kl{A^{q+1}+A^qB-B^{q+1}} \cdot Y + A^{q+1}B\tag{\ref{f3}*}\label{f3*}\\
0 &= \kl{A^{q}-B^{q}} \cdot Y + A^{q+1}-B^{q+1}\tag{\ref{f4}*}\label{f4*}
\end{alignat}
Clearly, the solutions to (\ref{f1}), (\ref{f2}), (\ref{f3}) and
(\ref{f4}) also solve the respective equations above. 
As $N(A)=N(B)=1$, the coefficient of $Y^2$ in (\ref{f2*}) and (\ref{f3*})	is clearly non-zero. If the coefficient of $Y$ in  (\ref{f1*}) or (\ref{f4*})  would be zero, then by raising them to the $(q^2-q+1)$th and $q^2$th power respectively, we would get $A^2=B^2$ or $A=B$. This would mean that we either have $A=B$ or $A=-B$, which are both impossible as $\frac{A}{B}\notin \fq$ by assumption. Hence (\ref{f1*}) and (\ref{f4*}) are linear and (\ref{f2*}) and (\ref{f3*}) are quadratic equations in $Y$. 

The linear equations have one solution each, namely
$X_1=\displaystyle\frac{-A^{q+1}B+AB^{q+1}}{A^{q+1}-B^{q+1}}$ solves
(\ref{f1*}) and $X_2=\displaystyle\frac{A^{q+1}-B^{q+1}}{-A^q+B^q}$
solves (\ref{f4*}). Bearing in mind that $N(A)=N(B)=1$, it is a
straightforward calculation to show that $X_1$ and $X_2$ are also
solutions to (\ref{f1}) and (\ref{f4}) respectively (for details see
Claim \ref{linear_eqs} in the Appendix).

By \Cref{lemma:rootsQuadratic}, the quadratic equations have $1+\eta_{\fqq 3}(D_1)$ and $1+\eta_{\fqq 3}(D_2)$ solutions in $\fqq 3$ respectively, where
\begin{align*}
D_{1}  &= \kl{B^qA+B^{q+1}-A^{q+1}}^2 - 4 \cdot B^q \cdot AB^{q+1}= h(B,A)^2 - 4 \cdot B^q \cdot A \cdot h(A,B) \text{ and}\\
D_{2} &= \kl{A^qB+A^{q+1}-B^{q+1}}^2 - 4 \cdot A^q \cdot BA^{q+1} = h(A,B)^2 - 4 \cdot A^q \cdot B \cdot h(B,A)
\end{align*}
are their respective discriminants. A somewhat longer but still
straightforward calculation, which uses $N(A)=N(B)=1$ and the fact
that the solutions are from the field $\mathbb F_{q^3}$, shows that all
$\fqq 3$-solutions of (\ref{f2*}) and (\ref{f3*})
are solutions of (\ref{f2}) and (\ref{f3}), respectively (for the details
see Claim \ref{quadratic_eqs} in the Appendix). 
	
Accordingly the number of solutions to the original system is
\begin{equation*}
\ubr{(\ref{f1*})}1 + \ubr{(\ref{f2*})}{1+\eta_{\fqq 3}(D_1)} + \ubr{(\ref{f3*})}{1+\eta_{\fqq 3}(D_2)} + \ubr{(\ref{f4*})}1 = 4 + \eta_{\fqq 3}(D_1) + \eta_{\fqq 3}(D_2).
\end{equation*}
As $\frac{A}{B}\in S_4(1,-1)$, by \Cref{lemma:S4bijection_a} we either have $h(\frac AB, 1)=0$ or $h(1,\frac AB)=0$. As $h$ is a homogeneous polynomial, this is equivalent to $h(A,B) = 0$ or $h(B,A)=0$. 

First assume that we have $h(A,B) = 0$. The discriminants $D_1$ and $D_2$ can now be simplified significantly.
\begin{equation*}
D_{1}  = h(B,A)^2\quad \text{ and }\quad D_{2} = - 4 \cdot A^q \cdot B \cdot h(B,A).
\end{equation*}
$D_{1}$ is clearly a square and so $\eta_{\fqq 3}(D_1)=1$. Using the
assumptions on $A$ and $B$ we get that not just $A$ and $B$ but $A^q$,
$B^q$ and $-(A+B)$ are also elements of norm $1$. For the latter, one
can transform $N\left(\frac{A}{B} +1 \right) =-1$.
Consequently from \Cref{lemma:NormProperties}(e) it follows that they
are squares in $\fqq{3}$, and therefore the value of the quadratic
character does not change with the inlcusion and removal of these factors.
\begin{align*}
\eta_{\fqq 3}(D_{2}) &= \eta_{\fqq 3}\big(- 4 \cdot A^q \cdot B \cdot h(B,A)\big)=\eta_{\fqq 3}(4) \cdot \eta_{\fqq 3}(A^q) \cdot \eta_{\fqq 3}(B) \cdot \eta_{\fqq 3}\big(-h(B,A)\big)\\
&= \frac{\eta_{\fqq 3}\big(-(A+B)\big)}{\eta_{\fqq 3}(B^q)} \cdot \eta_{\fqq 3}\big(-h(B,A)\big)= \eta_{\fqq 3}\kl{\frac{A+B}{B^q}\cdot h(B,A)}
\end{align*}
Now as $h(A,B)=0$ we have
\begin{align*}
\eta_{\fqq 3}(D_{2}) 
&= \eta_{\fqq 3}\kl{\frac A{B^q} \cdot \big(h(B,A) - h(A,B)\big) + \kw{B^{q-1}}\cdot h(B,A)}= \eta_{\fqq 3}\kl{A^2+AB+B^2}.
\end{align*}
In the other case, when $h(B,A)=0$, because of symmetry we obtain
$\eta_{\fqq 3}(D_{2})=1$ and $\eta_{\fqq 3}(D_{1})=\eta(A^2+AB+B^2)$. 

Hence, in both cases the number of solutions to the original system is
\begin{equation*}
5 + \eta_{\fqq 3}\kl{A^2+AB+B^2},
\end{equation*}
as desired.
\end{proof}

Before our construction we still need one final lemma, which will enable us
to control the character value $\eta_{\fqq 3}\kl{A^2+AB+B^2}$ by
transferring the problem to the less esotheric realm of group $G$.
Through this step we will be able to ensure that this character value 
is occasionally $1$ and hence an appropriately chosen corresponding 
quadruple does have six common neighbors.

\begin{lemma}\label{lemma:S4bijection_c}
If $p\neq 2,3$, then for every $D\in S_4(1,-1)\setminus\fq$, we have 
\begin{equation*}
\eta_G(C(D)) = \eta_{\fqq 3}\kl{D^2+D+1}, 
\end{equation*}
that is, $C(D)$ is a square in $G$ if and only if $D^2+D+1$ is a square in $\fqq 3$.
\end{lemma}

\begin{proof}
Let $r\in \fqc$ and $s \in\fqc$ be a square root of
$D^2+D+1$ and $C(D)$, respectively. 
Then $\eta_{\fqq{3}}(D^2+D+1)=1$ if and only if $r^{q^3-1}=1$ 
(i.e., $r\in \fqq{3}$). Similarly $\eta_G(C(D))=1$
if and only if  $s^{3(q-e_q)}=1$ (i.e., $s\in G$).

Recall that $a$ and $a^{-1}$ are the non-trivial third roots of unity
in $\fqq{2}$,
which exist since $p\neq 3$. They are the roots of the polynomial $X^2+X+1$, therefore
\begin{equation*}
r^2 = D^2+D+1 = (D-a^{-1}) \cdot (D-a)= C(D) \cdot (D-a)^2 = s^2\cdot (D-a)^2.
\end{equation*}
%
Using that  $q^3-1$ is even (as $p\neq 2$,),  $D^{q^3}=D$,
$a^q=a^{e_q}$ and ${e_q}^3=e_q$, we have
\begin{equation*}
r^{q^3-1}= \kl{s \cdot (D-a)}^{q^3-1}=s^{q^3-1}\cdot\frac{D^{q^3}-a^{q^3}}{D-a}= s^{q^3-1} \cdot \frac{D-a^{e_q}}{D-a}.
\end{equation*}
The latter fractional expression is $1$ if $e_q=1$, and it is $C(D)=s^2$ if $e_q=-1$ (note that as $p\neq 3$, $e_q=0$ is not possible), so
\begin{equation*}
r^{q^3-1} = s^{q^3-e_q}= \kl{s^{q-e_q}} ^ {q^2+e_q q+1}.
\end{equation*}
Since $q\equiv e_q \pmod 3$, $q$ is odd and not divisible by $3$, $q^2+e_q q+1$ is divisible
by 3 and odd. Furthermore $3(q-e_q)$ is even, so 
\begin{align*}
r^{q^3-1}&=\kl{s^{3(q-e_q)}}^{\frac{q^2+e_q q+1}{3}}=\kl{(s^2)^{\frac{3(q-e_q)}{2}}}^{\frac{q^2+e_q q+1}{3}}\\
&=\kl{C(D)^{\frac{3(q-e_q)}{2}}}^{\frac{q^2+e_q q+1}{3}}=C(D)^{\frac{3(q-e_q)}{2}}=s^{3(q-e_q)},
\end{align*}
where we used that by Lemma~\ref{lemma:S4bijection_b} $C(D)\in G$ so $C(D)^{\frac{3(q-e_q}{2})}=\pm 1$, and that $\frac{q^2+e_q q+1}{3}$ is odd.  In particular this implies that we have $r^{q^3-1}=1$ if and only if $s^{3(q-e_q)}= 1$, as wanted.
\end{proof}

We are now ready to construct $K_{4,6}$ subgraphs in $\NG(q,4)$.

\begin{proof}[Proof of Theorem~\ref{thm:K46}]
Consider a quadruple of vertices of $\NG(q,4)$  of the form
\begin{equation*}
 U = \men{(A_i,a_i)}_{i\in [4]} = \men{ (1,1), \kl{\kw{A+1}, -N\kl{\kw{A+1}}}, \kl{\kw{B+1}, -N\kl{\kw{B+1}}}, (0,1) },
\end{equation*}
where $A,B\in \fqq{3} \setminus \{ -1\}$, $A\neq B$. Recall the definition of $B_i(U)$ and
$H(U)$ from Lemma \ref{lemma:DegreeHelper}. The lemma then implies that
\begin{equation*}
\deg(U) = \absw{ \anz{H(U)} - 1 &\text{if } a_1=a_2=a_3=a_4 \\ \anz{H(U)} &\text{otherwise}},
\end{equation*}
where $H(U)$ in this case is the set of those $X\in \fqq 3$ for which
\begin{equation*} 
N(X+1) = 1, \quad N(X+1+A) = -1,  \quad N(X+1+B) = -1.
\end{equation*}
%

Substituting $Y=X+1$ in the above system we arrive at the system from
Lemma \ref{lemma:SpecialQuadruples}. 

Our plan is to apply Lemma \ref{lemma:SpecialQuadruples}, so we will
select $A, B \in \fqq{3}\setminus \{ -1\}$ such that $N(A)=N(B)=1$ and $\frac{A}{B} \in
S_4(1,-1) \setminus \fq$. In order to have six solutions, we will make sure that $C\left(\frac{A}{B}\right)$ is a square and apply
Lemma~\ref{lemma:S4bijection_c} with $D=\frac{A}{B}$.
Finally, in order to have $\deg(U)=|H(U)|=6$ in Lemma~\ref{lemma:DegreeHelper}, we will select $B$ such
that not only $N(B)=1$, but also $N(B+1)\neq -1$, so $a_3\neq 1=a_4$.

For this latter condition we select $B=1$, so $N(B+1)= N(2)=8\neq -1$,
so we do have $\deg(U) = \anz{H(U)}$.

For the selection of $A$ we fix a  generator $g$ of the cyclic group $G$. Then clearly
$g^2\in G\setminus G^3$, otherwise the order of $g$ would be
$2(q-e_q) < 3(q-e) = |G|$, contradicting the fact that $g$ generates
$G$. 
Consequently Lemma~\ref{lemma:S4bijection_b} ensures that there exists
$A:=C^{-1}(g^2) \in S_4(1,-1)\setminus \mathbb{F}_q$.

We have then $N(A)=1$, since $A\in S_4(1,-1)$. 
Furthermore $\frac{A}{B}\in S_4(1,-1)$, since $N(\frac{A}{B})=N(A)=1$
and $N(\frac{A}{B}+1)=N(A+1)=-1$, since $A\in S_4(1,-1)$.
Finally $\frac{A}{B}=A\notin\fq$, so the chosen elements $A$ and $B$ meet
all the conditions of
\Cref{lemma:SpecialQuadruples} and we have $\deg(U) = \anz{H(U)} = 5+\eta_{\fqq 3}\kl{A^2+A+1}$.

Finally, by Lemma~\ref{lemma:S4bijection_c}
$\eta_{\fqq{3}}\kl{A^2+A+1} =\eta_G(C(A))=\eta_G(g^2)=1$, hence
$\deg(U)$ is indeed six. 

To complete the proof of Theorem~\ref{thm:K46}, we still need to make
sure that the $24$ incidences, the existence of which we have just
proved, gives rise to an actual copy of $K_{4,6}$ in $\NG(q,4)$. 
The problem could be that the neighbourhood $\mathcal{N}(U) = \{(W_j,
w_j) : j\in [6]\}$ intersects $U = \{(A_i, a_i) : i \in [4]\}$. 
To overcome this, we apply certain transformations to create many new
vertex sets from $U$, each with common degree six. For any 
$\alpha \in \fqq{3}$, $\beta \in \fqq{3}^*$, and $c\in \fqu$ we define
\begin{equation*}
U^{\alpha, \beta, c} : = \big\{ (\beta A_i+\alpha, N(\beta)ca_i) : i=1,2,3,4\big\}.
\end{equation*}
We see  that $|U^{\alpha,\beta,c}|=4$ because the transformation we
performed on the first coordinates is bijective. 
Then the neighborhood of $U^{\alpha, \beta, c}$ can be expressed by
\begin{equation*}
\mathcal{N}(U^{\alpha, \beta, c}) := \big\{ \beta W - \alpha, c^{-1}w) : (W, w) \in
\mathcal{N}(U) \big\},
\end{equation*}
as the incidencies easily follow from $N(A_i+W)=a_iw$ for every $i\in
[4]$ and $(W,w) \in \mathcal{N}(U)$. 
Consequently $\deg\left(U^{\alpha, \beta, c}\right)=6$  for every
choice of the parameters $\alpha \in \fqq{3}$, $\beta \in \fqq{3}^*$,
and $c\in \fqu$. 

Now for every choice of $\beta$ and $c$, and
adjacency $(A_i,a_i)\sim(W_j, w_j)$, there is a
unique ``forbidden translation'' $\alpha$, namely 
$\alpha= \frac{\beta(W_j-A_i)}{2}$, for which the images of $A_i$ and $W_j$ are equal and hence which might make 
the adjacency into a loop. 
In conlcusion there are at least $(q^3-1)(q-1)(q^3-24)$ sets
$U^{\alpha, \beta, c}$ of size four, which are disjoint from their
respective common neighborhoods that have size six. 
Each of these sets gives rise to a different copy of $K_{4,6}$,
because among the maps of the form $X\mapsto \beta X+\alpha$ only the
identity map stabilizes a set $\{1,\frac{1}{A+1},\frac{1}{2},0\}$ with
$A\not\in \mathbb F_q$. This proves our theorem.
\end{proof}

\section{Generic quadruples}\label{sec:quadruples}
In this section we prove part $(c)$ of Theorem \ref{thm:main} by giving a relatively elementary argument using
resultants. 

\begin{proof}[Proof of Theorem~\ref{thm:main}$(c)$]
Let $T=\{(A_i,a_i): i\in[4]\}\subseteq L$ be a generic vertex set of
size four in $\NG(q,t)$. As before, for $i\in[3]$ put $\displaystyle
B_i=\frac{1}{A_i-A_4}$ and $\displaystyle b_i=\frac{a_i}{a_4}\cdot
N(B_i)$. 
By Lemma \ref{lemma:DegreeHelper} we have that $\deg(T)\leq |H(T)|$, where $H(T)$ is the set of solutions to the system
\begin{equation}\label{original}
N(Y+B_i)=b_i,\  i=1,2,3.
\end{equation}
Consider the equation system
\begin{equation}\label{eq:newsystem}
f_i(Y_1,\dots,Y_{t-1})=\prod_{j=1}^{t-1}(Y_j-C_{i,j})-b_i=0,\ i=1,2,3,
\end{equation}
where $C_{i,j}=-B_i^{q^{j-1}}$, $i=1,2,3$, $j=1,\dots,t-1$.

For every solution $Y\in \fqt$ of \eqref{original} the vector 
$(Y, Y^q, \ldots , Y^{q^{t-2}}) \in \fqt^{t-1}$ is a solution of
  \eqref{eq:newsystem}. These are all distinct, hence it will be
  enough to show that \eqref{eq:newsystem} has at most 
$6(q^{t-4} + \cdots + q + 1)$ solutions. 

For polynomials $p(z)=p_nz^n+\cdots p_1z+p_0$ and $r(z)=r_mz^m+\cdots r_1z+r_0$ of degree $n$ and $m$ respectively, in the variable $z$ over some field $\mathbb{F}$, their \emph{Sylvester matrix} is the $(n+m)\times(n+m)$ matrix $\text{Syl}(p,r)=\{s_{i,j}\}_{i,j\in[n+m]}$  with entries
\begin{equation*}
s_{i,j}=\absw{p_{n+i-j}&\text{if } 1\leq i \leq m \\ r_{i-j} & m+1
  \leq i \leq m+n\\
0 & \mbox{ otherwise.}}
\end{equation*}
For an example consider Figure \ref{fig:matrix}.

\begin{figure}[ht] 
$$\left(\begin{array}{ccccccc}
p_4 & p_3 & p_2 & p_1 & p_0 & 0 & 0\\
0 & p_4 & p_3 & p_2 & p_1 & p_0 & 0\\
0 & 0 & p_4 & p_3 & p_2 & p_1 & p_0\\
r_3 & r_2 & r_1 & r_0 & 0 & 0 & 0\\
0 & r_3 & r_2 & r_1 & r_0 & 0 & 0\\
0 & 0 & r_3 & r_2 & r_1 & r_0 & 0\\
0 & 0 & 0 & r_3 & r_2 & r_1 & r_0
\end{array}\right)$$
\caption{The Sylvester matrix for $n=4$ and $m=3$ \label{fig:matrix}}
\end{figure}
An important property of the Sylvester matrix is that the degree of
the greatest common divisor of $p$ and $r$ is
$n+m-\mbox{rank}(\text{Syl}(p,r))$, in particular if $p$ and $r$ have
a common root, then the determinant of $\text{Syl}(p,r)$, also called
the \emph{resultant} of $p$ and $r$, is $0$. This holds true even if $p_n=0$ or $r_m=0$, that is, when $n$ and $m$
are only upper bounds on the degree of $p$ and $r$. 
(See e.g. \cite{K}.)
Now if $p$ and $r$ are multivariate polynomials in the variables $Y_1,\dots,Y_n$ over some field $F$, then we can write them as univariate polynomials in $Y_n$, and consider their Sylvester matrix (now with entries from $\mathbb{F}[Y_1,\dots,Y_{n-1}]$). We will call the determinant of this matrix the \emph{Sylvester resultant} of $p$ and $r$ with respect to $Y_n$, and denote it by $\mbox{Res}_{Y_n}(p,r)$. Note that $\mbox{Res}_{Y_n}(p,r)$ is a polynomial in the variables $Y_1,\dots,Y_{n-1}$. From the above property of the Sylvester matrix it follows that if $(C_1,\dots,C_n)$ is a common root of $p$ and $r$, then $(C_1,\dots,C_{n-1})$ is a root of $\mbox{Res}_{Y_n}(p,r)$.

Let us now return to the polynomials $f_1,f_2,f_3 \in \fqt[Y_1, \ldots , Y_{t-1}]$. Our plan is to compute
$g_i=\mbox{res}_{Y_{t-1}}(f_i,f_3)$ for $i=1,2$ and then
$g=\mbox{Res}_{Y_{t-2}}(g_1,g_2)$. Then by the above, if $(C_1, \ldots
, C_{t-1})\in \fqt^{t-1}$ is a common root of $f_1, f_2,$ and $f_3$, then $(C_1,
\ldots , C_{t-3})\in \fqt^{t-3}$ is a root of $g$.

For the computation we introduce
\begin{equation*}
h_i=h_i(Y_1,\dots,Y_{t-3})=\prod_{j=1}^{t-3}(Y_j-C_{i,j})
\end{equation*}
for $i=1,2,3$, and rewrite $f_i$ as
univariate linear polynomials in $Y_{t-1}$:
\begin{equation*}
f_i=\left(h_i\cdot \left(Y_{t-2}-C_{i,t-2}\right)\right)\cdot Y_{t-1}-\left(h_i\cdot C_{i,t-1}(Y_{t-2}-C_{i,t-2})+b_i\right).
\end{equation*}
Then for $i=1,2$ we have
\begin{align*}
g_i=
&\mbox{Res}_{Y_{t-1}}(f_i,f_3)=
\left|\begin{array}{cc}
h_i\cdot (Y_{t-2}-C_{i,t-2}) & -\left\{h_i\cdot C_{i,t-1}(Y_{t-2}-C_{i,t-2})+b_i\right\} \\
h_3\cdot (Y_{t-2}-C_{3,t-2}) & -\left\{h_3\cdot C_{3,t-1}(Y_{t-2}-C_{3,t-2})+b_3\right\}
\end{array}\right|\\
=&\left|\begin{array}{cc}
h_i\cdot (Y_{t-2}-C_{i,t-2}) & -h_i\cdot C_{i,t-1}(Y_{t-2}-C_{i,t-2}) \\
h_3\cdot (Y_{t-2}-C_{3,t-2}) & -h_3\cdot C_{3,t-1}(Y_{t-2}-C_{3,t-2})
\end{array}\right|+
\left|\begin{array}{cc}
h_i\cdot (Y_{t-2}-C_{i,t-2}) & -b_i \\
h_3\cdot (Y_{t-2}-C_{3,t-2}) & -b_3
\end{array}\right|\\
=&h_i\cdot h_3\cdot(Y_{t-2}-C_{i,t-2})(Y_{t-2}-C_{3,t-2}) 
\left|\begin{array}{cc}
1 & -C_{i,t-1} \\
1 & -C_{3,t-1}
\end{array}\right|
-h_i\cdot b_3(Y_{t-2}-C_{i,t-2})\\
&+h_3\cdot b_i(Y_{t-2}-C_{3,t-2}).
\end{align*}
That is, $g_i=c_{i,2}Y_{t-2}^2+c_{i,1}Y_{t-2}+c_{i,0}$ is a quadratic
polynomial in $Y_{t-2}$ with coefficients
\begin{align*}
c_{i,2}&=h_i\cdot h_3\cdot \left|\begin{array}{cc}
1 & -C_{i,t-1} \\
1 & -C_{3,t-1}
\end{array}\right|,\\
c_{i,1}&=-h_i\cdot h_3\cdot (C_{i,t-2}+C_{3,t-2})\left|\begin{array}{cc}
1 & -C_{i,t-1} \\
1 & -C_{3,t-1}
\end{array}\right|-h_i\cdot b_3+h_3\cdot b_i,\\
c_{i,0}&=h_i\cdot h_3\cdot C_{i,t-2}C_{3,t-2}\left|\begin{array}{cc}
1 & -C_{i,t-1} \\
1 & -C_{3,t-1}
\end{array}\right|+h_i\cdot b_3 C_{i,t-2}-h_3\cdot b_i C_{3,t-2}.
\end{align*}
Hence the resultant of $g_1$ and $g_2$ is a four-by-four determinant.
\begin{equation*}
g=\mbox{Res}_{Y_{t-2}}(g_1,g_2)=\left|\begin{array}{cccc}
c_{1,2} & c_{1,1} & c_{1,0} & 0 \\
0 & c_{1,2} & c_{1,1} & c_{1,0} \\
c_{2,2} & c_{2,1} & c_{2,0} & 0 \\
0 & c_{2,2} & c_{2,1} & c_{2,0}
\end{array}\right|.
\end{equation*}
Note that each $c_{i,j}$ is a quadratic polynomial in each of the
variables $Y_1,\dots,Y_{t-3}$.  In particular the degree of $g$ in any
of the variables is at most $8$. It turns out that this bound can be reduced. 

\begin{lemma}
For $1\leq a\leq t-3$ the coefficient of $Y_a^8$ in $g$ 
 is $0$.
\end{lemma}
\begin{proof}
The coefficient in question is clearly the determinant we get by replacing $c_{i,j}$ everywhere in the determinant formula for $g$ with the coefficient of $Y_a^2$ in it. As
\begin{align*}
\mbox{coeff}(Y_a^2,c_{i,2})&=\frac{h_i}{Y_a-C_{i,a}}\cdot \frac{h_3}{Y_a-C_{3,a}}\cdot \left|\begin{array}{cc}
1 & -C_{i,t-1} \\
1 & -C_{3,t-1}
\end{array}\right|,\\ 
\mbox{coeff}(Y_a^2,c_{i,1})&=-\frac{h_i}{Y_a-C_{i,a}}\cdot \frac{h_3}{Y_a-C_{3,a}}\cdot (C_{i,t-2}+C_{3,t-2}) \left|\begin{array}{cc}
1 & -C_{i,t-1} \\
1 & -C_{3,t-1}
\end{array}\right|\\
\mbox{coeff}(Y_a^2,c_{i,0})&=\frac{h_i}{Y_a-C_{i,a}}\cdot \frac{h_3}{Y_a-C_{3,a}}\cdot C_{i,t-2}C_{3,t-2} \left|\begin{array}{cc}
1 & -C_{i,t-1} \\
1 & -C_{3,t-1}
\end{array}\right|,
\end{align*}
we have
\begin{equation*}
\mbox{coeff}(Y_a^8,g)=\left|\begin{array}{cccc}
\mbox{coeff}(Y_a^2,c_{1,2}) & \mbox{coeff}(Y_a^2,c_{1,1}) & \mbox{coeff}(Y_a^2,c_{1,0}) & 0 \\
0 & \mbox{coeff}(Y_a^2,c_{1,2}) & \mbox{coeff}(Y_a^2,c_{1,1}) & \mbox{coeff}(Y_a^2,c_{1,0}) \\
\mbox{coeff}(Y_a^2,c_{2,2}) & \mbox{coeff}(Y_a^2,c_{2,1}) & \mbox{coeff}(Y_a^2,c_{2,0}) & 0 \\
0 & \mbox{coeff}(Y_a^2,c_{2,2}) & \mbox{coeff}(Y_a^2,c_{2,1}) & \mbox{coeff}(Y_a^2,c_{2,0})
\end{array}\right|
\end{equation*}
\begin{equation*}
=\left(\frac{h_1}{Y_a-C_{1,a}}\right)^2\left(\frac{h_2}{Y_a-C_{2,a}}\right)^2 \left(\frac{h_3}{Y_a-C_{3,a}}\right)^4 \left|\begin{array}{cc}
1 & -C_{1,t-1} \\
1 & -C_{3,t-1}
\end{array}\right|^2 \left|\begin{array}{cc}
1 & -C_{2,t-1} \\
1 & -C_{3,t-1}
\end{array}\right|^2\cdot D,
\end{equation*}
where
\begin{equation*}
D=\left|\begin{array}{cccc}
1 & -(C_{1,t-2}+C_{3,t-2}) & C_{1,t-2}C_{3,t-2} & 0 \\
0 & 1 & -(C_{1,t-2}+C_{3,t-2}) & C_{1,t-2}C_{3,t-2} \\
1 & -(C_{2,t-2}+C_{3,t-2}) & C_{2,t-2}C_{3,t-2} & 0 \\
0 & 1 & -(C_{2,t-2}+C_{3,t-2}) & C_{2,t-2}C_{3,t-2}
\end{array}\right|.
\end{equation*}
Note that $D$ is just the Sylvester resultant of the two quadratic
univariate polynomials $(Y-C_{1,t-2})(Y-C_{3,t-2})$ and
$(Y-C_{2,t-2})(Y-C_{3,t-2})$. However these two have $C_{3,t-2}$ as a
common zero and hence their Sylvester resultant is $0$. This implies
that  the coefficient of $Y_a^8$ in $g$ is $0$.
\end{proof}

To reduce the effective degree of $g$ further, observe that $h_3$ can be factored
out from both $c_{2,2}$ and $c_{1,2}$, which are the non-zero entries of the
first column of the determinant defining $g$, hence $g = h_3\cdot g^*$
for some polynomial $g^*\in \fqt[Y_1, \ldots , Y_{t-3}]$.
Since $h_3$ is linear in each variable, the degree of $g^*$ in every
variable is at most six.

If $(C_1,\dots,C_{t-1})$ is a common zero of $f_1, f_2,$ and $f_3$,
then, as the $b_i$s are non-zero, we have $C_j\neq C_{i,j}$, for $i\in [3]$ and $j\in [t-1]$.  
In particular $h_3(C_1,\dots,C_{t-3})\neq 0$. On the other hand, by the
properties of the Sylvester resultant, we must have
$g(C_1,\dots,C_{t-3})=0$. This implies that $g^*(C_1,\dots,C_{t-3})=0$. 

Denote by $\tilde{g}$ the univariate polynomial that we obtain by
substituting $Y_i=Y^{q^{j-1}}$ in $g^*$  for $j\in [t-3]$. By the
degree bounds on $g^*$ we get that the degree of $\tilde{g}$ is at
most 
$6(1+q+q^2+\cdots +q^{t-4})$, in particular it has at most that many
roots. 
Now if $X$ is a solution to the original system (\ref{original}), then
$(X,X^q,\dots,X^{q^{t-2}})$ is a common root of the $f_i$s, hence
$(X,X^q,\dots,X^{q^{t-4}})$ is a root of $g^*$ and so $X$ is a root of
$\tilde{g}$. 
Consequently the number of solution to (\ref{original}) is also bounded by $6(1+q+q^2+\cdots +q^{t-4})$.
\end{proof}

\section{Applications}\label{sec:applications}

In this section, as an application of Theorem \ref{thm:main} we prove Theorem
\ref{thm:subgraph}.

\begin{proof}[Proof of Theorem \ref{thm:subgraph}] We start the proof by introducing some notation. Denote by $\Delta_d(q,t)$ and $\delta_d(q,t)$, respectively, the largest and smallest possible common degree of a generic $d$-tuple of vertices in the projective norm graph $\NG(q,t)$. For $d=0$, we set $\Delta_0(q,t) = \delta_0(q,t) = |V(\NG(q,t)|$.

Now let $H$ be a simple $\ell$-degenerate graph and suppose that $t\geq 3$. To simplify notation put $v=v(H)$ and $m=e(H)$. Further let $v_1,\dots,v_v$ be an ordering of the vertices of $H$ witnessing its $\ell$-degeneracy, i.e. every vertex $v_i$ has at most $\ell$ neighbours in $\{v_1,\dots,v_{i-1}\}$. For $1\leq i\leq v$ put $\mathcal{N}_i=\mathcal{N}(v_i)\cap \{v_1,\dots,v_{i-1}\}$ and $d_i=|\mathcal{N}_i|$, in particular $\mathcal{N}_1=\emptyset$ and $d_1=0$. With this notation for our ordering we have $d_i\leq \ell$ for $1\leq i\leq v$. 

To count the number of labelled copies of $H$ in $\NG(q,t)$ we will embed the vertices of $H$ into $\NG(q,t)$ one-by-one according the above order. Suppose we have already embedded $v_1,\dots,v_{i-1}$. To embed $v_i$, we have to choose a vertex from the common neighbourhood of the image
$T_i$ of $\mathcal{N}_i$ under this embedding. As $T_i$ is of size $d_i$, it has at most $\Delta_{d_i}(q,t)$ common neighbours in $\NG(q,t)$, so we have at most $\Delta_{d_i}(q,t)$ choices for $v_i$.  Accordingly 
\begin{equation*}
X_H(\NG(q,t))\leq \prod_{i=1}^v \Delta_{d_i}(q,t).
\end{equation*}
To obtain a similar lower bound we can repeat the same argument with the extra condition that during the embedding we want every possible set of already embedded vertices of size at most $\ell$ to be generic. We will achieve this simply by mapping the vertices of $H$ each time to a vertex of $\NG(q,t)$ with a first coordinate different from all the previous ones.

So suppose that we have already embedded $v_1,\dots,v_{i-1}$ with the desired property. To embed $v_i$, we have to choose a vertex from the
common neighbourhood of the image $T_i$ of $\mathcal{N}_i$ under this embedding whose first coordinate is different from those of $v_1,\dots,v_{i-1}$. As $T_i$ is now a generic set of size $d_i$, it has at least $\delta_{d_i}(q,t)$ common neighbours. To maintain our extra condition, when choosing the image of $v_i$ we have to exclude the common neighbours with first coordinate equal to the first coordinates of the previously selected ones. If $d_i=0$ then this means that we have to exclude $(i-1)(q-1)$ vertices, but there still will be at least $\delta_{0}(q,t)-(i-1)(q-1)\geq \delta_{0}(q,t)-vq$ candidates for the image of $v_i$. If $d_i>0$, then $T_i$ cannot contain two vertices with the same first coordinate, so for every previously selected vertex we have to exclude at most one vertex from $T_i$. Therefore there still will be at least $\delta_{d_i}(q,t)-(i-1)\geq \delta_{d_i}(q,t)-v$ candidates for the image of $v_i$. Accordingly we obtain that 
\begin{equation*}
X_H(\NG(q,t))\geq \prod_{i=1}^v \left(\delta_{d_i}(q,t)-v\chi_i\right),
\end{equation*}
where $\chi_i=q$ if $d_i=0$ and $\chi_i=1$ otherwise.

Now to finish the proof of Theorem \ref{thm:subgraph} we will consider two cases. 

First suppose $\ell=3$ and $t\geq 5$ or $\ell=2$ and $t\geq 3$. In both cases by Theorem~\ref{thm:main} we know that there exists a positive constant $C$ such that for all $d\leq \ell$ we have
\begin{equation}\label{deviation}
|\Delta_{d}(q,t)-q^{t-d} |,|\delta_{d}(q,t)-q^{t-d}|\leq Cq^{t-d-\frac{1}{2}}.
\end{equation}
Recall that by the construction of the order $d_i\leq \ell$ for $i\in [v]$, hence using (\ref{deviation}) we get
\begin{align*}
X_H(\NG(q,t))&\leq \prod_{i=1}^v \Delta_{d_i}(q,t) \leq \prod_{i=1}^v \left(q^{t-d_i}+Cq^{t-d_i-\frac{1}{2}}\right)=\left(\prod_{i=1}^v q^{t-d_i}\right) \left(1+\frac{C}{\sqrt{q}}\right)^v\\
&=q^{t\cdot v-(d_1+\dots d_v)}\left(1+\frac{C}{\sqrt{q}}\right)^v=q^{t\cdot v-m}\left(1+\frac{C}{\sqrt{q}}\right)^v 
\leq q^{t\cdot v-m}\left(1+C'\frac{v}{\sqrt{q}}\right)
\end{align*}
for some appropriate positive constant $C'$, whenever $v=o(\sqrt{q})$. Similarly again using (\ref{deviation}) we get
\begin{align*}
X_H(\NG(q,t))&\geq \prod_{i=1}^v \left(\delta_{d_i}(q,t)-v\chi_i\right)
               \geq \prod_{i=1}^v \left(q^{t-d_i}-Cq^{t-d_i-\frac{1}{2}}-v\chi_i\right)\\&\geq \prod_{i=1}^v \left(q^{t-d_i}-C''q^{t-d_i-\frac{1}{2}}\right)
\end{align*}               
for some appropriate positive constant $C''\geq C$. Note that for all sets of parameters in the case $d_i=0$ we have $t-\frac{1}{2}\geq \frac{3}{2}$ and in the case $d_i>0$ we have $t-d_i-\frac{1}{2}\geq \frac{1}{2}$, hence whenever $v=o(\sqrt{q})$ then for given $C$ such a $C''$ really exists. Hence
\begin{align*}
X_H(\NG(q,t))&\geq \left(\prod_{i=1}^v q^{t-d_i}\right)\left(1-\frac{C''}{\sqrt{q}}\right)^v=q^{t\cdot v-m}\left(1-\frac{C''}{\sqrt{q}}\right)^v\geq q^{t\cdot v-m}\left(1-v\frac{C''}{\sqrt{q}}\right).
\end{align*}
The two bounds together give that $X_H(\NG(q,t))$ is asymptotically $q^{t\cdot v-m}$ as desired.

Finally suppose $\ell=3$ and $t=4$. In this case, according to Theorem~\ref{thm:main}, $\Delta_3(q,4)$ and $\delta_3(q,4)$ differ asymptotically by a factor of $2$, so the same proof only yields
\begin{equation*}
q^{t\cdot v-m}\big(1-o(1)\big)\leq X_H(\NG(q,4))\leq 2^{c(H)}q^{t\cdot v-m}\big(1+o(1)\big),
\end{equation*} 
where $c(H)$ is the minimum number of indicies with $d_i=3$ in any witnessing ordering of the vertices of $H$. Accordingly this shows that $X_H(\NG(q,4))=\Theta(q^{t\cdot v-m})$ for any $H$ with $v=o(\sqrt{q})$  and $c(H)$ bounded.

\end{proof}

\section{The automorphism group of projective norm graphs}\label{sec:auto}

In this section, we aim to prove Theorem \ref{thm:aut}. For the composition of some maps $\alpha$ and $\beta$ we fix the notation $\alpha\circ \beta$ and their order of action is understood as $(\alpha\circ \beta)(x)=\alpha(\beta(x))$.

\begin{proof}[Proof of Theorem \ref{thm:aut}]
We start the proof by showing that all the maps presented are really automorphisms of $\NG(q,t)$. For this first note that by \Cref{lemma:FiniteFieldBasics} the map $X\mapsto X^{p^i}$ is an automorphism of $\fqt$ for every $i\inn$, in particular it is bijective and is interchangeable with the field operations and the norm function. Accordingly for any $C\in\fqt^*$, $c\in\fqu$ and $i\inn$ we have
\begin{equation*}
(X,x) \sim (Y,y)\ \gdw\  N(X+Y) = xy\ \gdw\ N(X+Y)^{p^i} = (xy)^{p^i}
\end{equation*}
\begin{equation*}
\gdw\ N(X^{p^i}+Y^{p^i}) = x^{p^i} y^{p^i}\ \gdw N(C^2) N(X^{p^i}+Y^{p^i}) = \big(N(C)\big) ^2 x^{p^i} y^{p^i}
\end{equation*}
\begin{equation*}
\gdw\ N\big(C^2  X^{p^i}+C^2  Y^{p^i}\big) = \big(\pm N(C) x^{p^i}\big)\big(\pm N(C)  y^{p^i}\big)
\end{equation*}
\begin{equation*}
\gdw\ \big( C^2 X^{p^i},\pm N(C) x^{p^i} \big) \sim \big( C^2 Y^{p^i}, \pm N(C) y^{p^i}\big),
\end{equation*}
hence all maps presented for $q$ odd are really automorphisms of $\NG(q,t)$. For $q$ even $2A=0$ for every $A\in \fqt$ and hence for such an $A$, by continuing the previous series of equivalences, we additionally have
\begin{equation*}
( X,x ) \sim ( Y,y)\ \gdw\ \big( C^2 X^{p^i}+A,\pm N(C) x^{p^i} \big) \sim \big( C^2 Y^{p^i}+A, \pm N(C) y^{p^i}\big) 
\end{equation*}
which finishes this part for even $q$ as well. 

\medskip

Next we need to show that any automorphism in $\Aut(\NG(q,t))$ is of the given form. To do so we start by observing that any $\Phi\in \Aut(\NG(q,t))$ must act independently on the two coordinates.

\begin{lemma} \label{lemma:AutoDecomposition}
Let $q=p^{k}>2$ be a prime power, $t\geq 2$ an integer and $\Phi\in \Aut(\NG(q,t))$. Then there are permutations $\Psi:\fqt\to\fqt$ and $\psi:\fqu\to\fqu$ such that
\begin{equation*}
\Phi\big( (X,x)\big) = \big( \Psi(X), \psi(x) \big)
\end{equation*}
and $\Psi(-X)=-\Psi(X)$ for all $X\in \fqt$.
\end{lemma}

\begin{proof}
To start the proof observe that certainly there are maps $\Gamma: \fqt\times\fqu\to\fqt$ and $\gamma: \fqt\times\fqu\to\fqu$ such that
\begin{equation*}
\Phi\big( (X,x)\big) = \big( \Gamma(X,x), \gamma(X,x) \big).
\end{equation*}
Now define a set $S$ of vertices \emph{poor} if for any pair of vertices $u,v\in S$ we have $\displaystyle\deg(\{u,v\})<\frac{q^{t-1}-1}{q-1}$. Further, for $x\in \fqu$ we put
\begin{equation*}
S_x=\{(X,x)\ |\ X\in \fqt\}\subseteq V(\NG(q,t)).
\end{equation*}
Note that for every $x\in \fqt$ the set $S_x$ has size $q^{t-1}$  and according to Theorem~\ref{thm:main}(b) it is poor. On the other hand we also claim, that any poor set of vertices of size $q^{t-1}$ is of the form $S_x$ for some $x\in \fqt$. Indeed, let $S$ be a poor set of vertices of size $q^{t-1}$ and suppose that it is not of the given form. In this case there exist $(Y,y),(Y',y')\in S$ with $y\neq y'$. As $S$ is poor we must have $\deg(\{(Y,y),(Y',y')\})<\displaystyle\frac{q^{t-1}-1}{q-1}$. A $y\neq y'$, this is possible only if $Y=Y'$. Now, as there are at most $q^{t-1}-1\geq 2$ vertices with first coordinate equal to $Y$, there must exist a vertex $(X,x)\in S$ with $X\neq Y$. However in this case either $y$ or $y'$ is different from $x$ and so again by Theorem~\ref{thm:main}(b) either $(X,x)$ and $(Y,y)$ or $(X,x)$ and $(Y,y')$ have common degree $\displaystyle \frac{q^{t-1}-1}{q-1}$, contradicting with the poorness of $S$. 

However $\Phi$ must map poor sets to poor sets, so there exists some function $\psi:\fqu\rightarrow \fqu$ for which we have $\Phi(S_x)=S_{\psi(x)}$ and hence $\gamma(X,x)=\psi(x)$ for every $x\in \fqu$ and $X\in \fqt$. As $\Phi$ is surjective, so must be $\psi$, hence it is a permutation of $\fqu$.

Analogously as before, for $X\in \fqt$ we put
\begin{equation*}
\overline{S}_X=\{(X,x)\ |\ x\in \fqu\}.
\end{equation*}
Cleary $\overline{S}_X$ is also poor for every $X\in\fqt$ (as every pair in it is non-generic), hence so must be
\begin{equation*}
\Phi(\overline{S}_X)=\{\big(\Gamma(X,x),\psi(x)\big)\ |\ x\in \fqu\}.
\end{equation*}
However here all the second coordinates are again different, so again this is possible only if all the first coordinates are the same. This means that there is some function $\Psi:\fqt\rightarrow \fqt$ such that $\Gamma(X,x)=\Psi(X)$ for every $X\in \fqt $ and $x\in \fqu$ which is clearly surjective, as $\Phi$ is, and hence is a permutation of $\fqt$.

Now what remains to show is that $\Psi(-X)=-\Psi(X)$ for all $X\in \fqt$. For this we define two sets of vertices $S$, $S'$ adjacent if there are vertices $v\in S$ and $w\in S'$ such that $(v,w)$ is an edge. Now for $X\in \fqt$ the set $\overline{S}_X$ is adjacent to $\overline{S}_Y$ exactly if $Y\in \fqt\setminus\{-X\}$, in particluar $\overline{S}_X$ and $\overline{S}_{-X}$ are non-adjacent. However then so must be $\Phi(\overline{S}_X)=\overline{S}_{\Psi(X)}$ and  $\Phi(\overline{S}_{-X})=\overline{S}_{\Psi(-X)}$, which is possible only if $\{\Psi(X),\Psi(-X)\}=\{Y,-Y\}$ for some $Y\in \fqt$ and hence $\Psi(-X)=-\Psi(X)$ holds, as desired.
\end{proof}

To continue we will consider two cases and suppose first that $t>2$. To obtain further properties of $\Psi$ and $\psi$ in this case we will need a result of H.W.~Lenstra from \cite{L}.  For a field extension $L\supseteq K$ a bijection $\beta:L\to L$ is called a \emph{$K$-semilinear $L$-automorphism}, if $\beta(l_1 + l_2) = \beta(l_1) + \beta(l_2)$ for every $l_1,l_2\in L$ (that is, $\beta$ is an automorphism of the additive group of $L$), and there is a field automorphism $\gamma\in\Aut(K)$, such that $\beta(l_K\cdot l_L) = \gamma(l_K) \cdot \beta(l_L)$ for every $l_K\in K, l_L\in L$. If $K=L$, the notion of semilinearity simplifies significantly.

\begin{lemma} \label{lemma:Semilinearity}
Let $L$ be a field. Then $\beta$ is an $L$-semilinear $L$-automorphism if and only if there is some $\alpha\in\Aut(L)$ and $C\in L^*$ such that $\beta(x) = C \cdot \alpha(x)$ for every $x\in L$.
\end{lemma}

\begin{proof}
By definition, if $\beta$ is an $L$-semilinear $L$-automorphism, then there is some $\alpha \in \Aut(L)$ such that $\beta( x \cdot 1) = \alpha(x) \cdot \beta(1)$ for every $x\in L$, proving that $\beta$ is of the desired form with $C=\beta(1)$. Conversely, if $\beta$ is of the given form, then since $\alpha$ is an $L$-automorphism, it is also an automorphism of the additive group of $L$, and $\beta(l_K\cdot l_L) = C \cdot \alpha(l_K\cdot l_L) = \alpha(l_K) \cdot C \cdot \alpha (l_L) = \alpha(l_K) \cdot \beta(l_L).$
\end{proof}

\begin{theorem} [Lenstra ({\cite[Theorem 2]{L}})] \label{thm:Lenstra}
Let $F$ be a finite field, $E$ a non-trivial abelian group, $\Phi:
F^*\to E$ a surjective group homomorphism and $K\subset F$ the
subfield of $F$ generated by the kernel of $\Phi$. Then a permutation
$\rho:F\to F$ of $F$ satisfies that for some permutation $\kappa:E\to E$ 
\begin{equation*}
\Phi\big(\rho(x) - \rho(y)\big) = \kappa \big( \Phi(x-y) \big) \quad\forall x\neq y\in F
\end{equation*}
if and only if there is a $K$-semilinear $F$-automorphism $\beta$ and $b\in F$, such that
\begin{equation*}
\rho (x) = \beta(x) + b \quad\forall x\in F.
\end{equation*}
\end{theorem}

We will apply this theorem with $F=\fqt, E=\fqu$, $\Phi(x)=N(x)$ and accordingly $K$ being the subfield of $\fqt$ generated by $N^{-1}(1)$.

\begin{lemma} \label{lemma:KernelIsField}
Let $q=p^{k}$ be a prime power and $t>2$ an integer. Then there is no proper subfield of $\fqt$ containing $N^{-1}(1)$.
\end{lemma}

\begin{proof}
Assume to the contrary that $K$ is such a subfield of $\fqt=\mathbb{F}_{p^{k(t-1)}}$. Then by \Cref{lemma:FiniteFieldBasics} $K=\zb F_{p^s}$ for some proper divisor $s$ of $k(t-1)$, in particular $s\leq \frac{k(t-1)}{2}$ holds. On the other hand by \Cref{lemma:NormProperties}
\begin{equation*} 
p^s = \anz{K} \geq \anz{N^{-1}(1)} = \frac{q^{t-1}-1}{q-1} = \fsum{i=0}{t-2} q^i > q^{t-2} = p^{k(t-2)},
\end{equation*}
and hence $s>k(t-2)$.  Now together with the earlier inequality this gives $t-2 < \frac{s}{k} \leq \frac{t-1}{2}$, which is impossible for $t\geq 3$.
\end{proof}

Note that in our case \Cref{lemma:KernelIsField} just says that $K=\fqt$. Now to finish the proof of the first part of Theorem~\ref{thm:aut} in the case $t>2$ first recall that given $\Phi\in \Aut(\NG(q,t))$ by \Cref{lemma:AutoDecomposition} there exist permutations $\Psi:\fqt\to\fqt$ and $\psi\in\fqu\to\fqu$ such that
\begin{equation*}
\Phi\big( (X,x)\big) = \big( \Psi(X), \psi(x) \big)
\end{equation*}
and $\Psi(-X)=-\Psi(X)$ for every $X\in \fqt$. 

Now for any $X\neq Y\in\fqt$, the vertex $(X,N(X-Y))$ is adjacent to $(-Y,1)$, hence 
\begin{equation*}
\Phi\Big(\big(X,N(X-Y)\big)\Big)=\Big(\Psi(X),\psi\big(N(X-Y)\big)\Big)
\end{equation*}
and 
\begin{equation*}
\Phi\big((-Y,1)\big)=\big(\Psi(-Y),\psi(1)\big)=\big(-\Psi(Y),\psi(1)\big)
\end{equation*}
must be also adjacent meaning that we must have 
\begin{equation}\label{eq:AutProofNorm}
N\big(\Psi(X)-\Psi(Y)\big) = \psi\big((N(X-Y)\big) \psi(1).
\end{equation}
Note that $\kappa(x) = \psi(1) \psi(x)$ is a permutation of $\fqu$, and so we can apply \Cref{thm:Lenstra} with $F=K=\fqt$, $E=\fqu$, $\Phi=N$ and $\rho=\Psi$ to get that $\forall X\in\fqt$
\begin{equation*}
\Psi (X) = \beta(X) + A
\end{equation*}
for some $\fqt$-semilinear $\fqt$-automorphism $\beta$ and $A\in \fqt$. Now by \Cref{lemma:Semilinearity} we have $\beta(X)=C \alpha(X)$ for some $C\in\fqt^*$ and $\alpha\in\Aut(\fqt)$, which by \Cref{lemma:FiniteFieldBasics} has the form $\alpha(X)=X^{p^i}$ for some $i\in [k(t-1)]$. Hence
$\Psi(X) = C X^{p^i} + A$. However, as $\Psi(-1)=-\Psi(1)$, we must have $-C+A=-(C+A)$, implying $2A=0$. If $q$ is odd this is possible only if $A=0$ and if $q$ is even then this holds for every $A\in \fqt$.  

Now take any $x\in\fqu$ and evaluate Equation (\ref{eq:AutProofNorm}) for some $X\in N^{-1}(x)$ and $Y=0$ to get
\begin{equation*}
N(C X^{p^i} - C 0^{p^i}) = \psi\big(N(X-0)\big) \psi(1)\ \imp\ \psi(x) = N(C) \psi(1)^{-1} x^{p^i}.
\end{equation*}
Substituting $x=1$ we also obtain $\psi(1)^2=N(C)$. By \Cref{lemma:NormProperties}, $N(C)$ is a square if and only if $C$ is so, hence there exists $C'\in \fqt$ such that $C = C'^2$ and so $\psi(1)=\pm N(C')$. 

Hence with these choices of parameters for every $(X,x)\in \fqt\times \fqu$ we have 
\begin{equation*}
\Phi\big((X,x)\big) = \big( C'^2 X^{p^i}, \pm N(C') x^{p^i} \big)
\end{equation*}
if $q$ is odd and 
\begin{equation*}
\Phi\big((X,x)\big) = \big( C'^2 X^{p^i}+A, \pm N(C') x^{p^i} \big)
\end{equation*}
if $q$ is even, as desired.

\medskip

The only remaining case is when $t=2$. Note that then $t-1=1$ and we simply have $N(X)=X$ for every $X\in \fq$ and hence two vertices $(X,x),(Y,y)\in \fq\times\fq$ are adjacent if $X+Y=xy$. Given $\Phi\in \Aut(\NG)$ let $\Psi:\fqt\to\fqt$ and $\psi\in\fqu\to\fqu$ as before be the permutations guaranteed by \Cref{lemma:AutoDecomposition} such that
\begin{equation*}
\Phi\big( (X,x)\big) = \big( \Psi(X), \psi(x) \big)
\end{equation*}
and $\Psi(-X)=-\Psi(X)$ for every $X\in \fqt$. The latter in particular means that $\Psi(0)=-\Psi(0)$ and hence for $q$ odd we must have $\Psi(0)=0$. 

Let us now for $(X,x)\in \fq\times\fq$ define
\begin{equation*}
\widetilde{\Phi}\big((X,x)\big)=\big(\widetilde{\Psi}(X),\widetilde{\psi}(x)\big)=\left\{\begin{array}{cl}
\Big(\frac{1}{\psi(1)^2}\Psi(X),\frac{1}{\psi(1)}\psi(x)\Big) & \text{ if $q$ is odd}\\
\Big(\frac{1}{\psi(1)^2}\big(\Psi(X)-\Psi(0)\big),\frac{1}{\psi(1)}\psi(x)\Big) & \text{ if $q$ is even}
\end{array}\right..
\end{equation*}
The map $\widetilde{\Phi}$ is clearly also an automorphism of $\NG(q,2)$ with $\widetilde{\Psi}(-X)=-\widetilde{\Psi}(X)$ for every $X\in \fq$ and we also have $\widetilde{\Phi}\big((0,1)\big)=\big(\widetilde{\Psi}(0),\widetilde{\psi}(1)\big)=(0,1)$. 

As for every $X\in \fqu$ the vertices $(X,X)$ and $(0,1)$ are adjacent, so must be their images under $\widetilde{\Phi}$, implying that we have $\widetilde{\Psi}(X)=\widetilde{\Psi}(X)+\widetilde{\Psi}(0)=\widetilde{\psi}(X)\widetilde{\psi}(1)=\widetilde{\psi}(X)$ for every $X\in \fqu$.

Similarly, for $X\neq Y\in \fq$ the vertices $(X,X-Y)$ and $(-Y,1)$, and hence their images are also adjacent. Accordingly we get $\widetilde{\Psi}(X)-\widetilde{\Psi}(Y)=\widetilde{\Psi}(X)+\widetilde{\Psi}(-Y)=\widetilde{\psi}(X-Y)\widetilde{\psi}(1)=\widetilde{\Psi}(X-Y)$ for every $X\neq Y\in \fq$. However as we also have $\widetilde{\Psi}(0)=0$ this means that $\widetilde{\Psi}$ is an additive map on $\fq$.

Finally, for $X,Y\in \fqu$ looking at the images of the adjacent vertices $(XY,X)$ and $(0,Y)$ we obtain $\widetilde{\Psi}(XY)=\widetilde{\Psi}(XY)-\widetilde{\Psi}(0)=\widetilde{\psi}(X)\widetilde{\psi}(Y)=\widetilde{\Psi}(X)\widetilde{\Psi}(Y)$. As $\widetilde{\Psi}(0)=0$, the above clearly also holds when $X$ or $Y$ are $0$ which means that $\widetilde{\Psi}$ is also a multiplicative map on $\fq$ and hence is actually an automorphism of $\fq$. According to \Cref{lemma:FiniteFieldBasics} then $\widetilde{\Psi}(X)=X^{p^i}$ for some $i\in [k(t-1)]$.

In terms of $\Phi$ this just means that it has the form
\begin{equation*}
\Phi\big((X,x)\big)=\big(\Psi(X),\psi(x)\big)=\left\{\begin{array}{cl}
\Big(\psi(1)^2 X^{p^i},\psi(1)x^{p^i}\Big) & \text{ if $q$ is odd}\\
\Big(\psi(1)^2 X^{p^i}+\Psi(0),\psi(1)x^{p^i}\Big) & \text{ if $q$ is even.}
\end{array}\right.
\end{equation*}
To finish, just note that we have $N(\Psi(1))=\Psi(1)$, so this is really the form we wanted.

Now we turn our attention to the group struccture of $\Aut(\NG(q,t))$. To begin with, we define the following subgroups of $\Aut(\NG(q,t))$.
\begin{alignat*}1
\Aut_F &= \men{ \pi_i: (X,x) \mapsto (X^{p^i},x^{p^i}) \mid i\in [k(t-1)] }\isom Z_{k(t-1)} \\
\Aut_M &= \men{ \sigma_C: (X,x) \mapsto (C^2 X, N(C) x) \mid C\in\fqt^* }\\
&\isom \left\{\begin{array}{cl}
Z_{q^{t-1}-1}& \text{ if $q$ is even or both $q$ and $t-1$ are odd}\\
Z_{\frac{q^{t-1}-1}{2}} & \text{ if $q$ is odd and $t-1$ is even.}
\end{array}\right.
\end{alignat*}
In addition, for $q$ odd we also consider the subgroup
\begin{equation*}
\Aut_S = \men{ \omega_\ep: (X,x) \mapsto (X,\ep x) \mid \ep \in \men{-1,+1} }\isom Z_{2}
\end{equation*}
and for $q$ even the subgroup
\begin{equation*}
\Aut_L = \men{ \mu_A: (X,x) \mapsto (X+A, x) \mid A\in\fqt}\isom \kl{Z_p}^{k(t-1)}.
\end{equation*}
Now take $\Phi\in \Aut(\NG(q,t))$. When $q$ is odd then, according to the first part of Theorem~\ref{thm:aut}, there exists $C\in \fqt^*$, $i\in [k(t-1)]$ and $\ep\in\men{-1,+1}$ such that $\Phi\big((X,x)\big)=\big(C^2X^{p^i},\ep N(C)x^{p^i}\big)$ and hence $\Phi=\omega_\ep\circ  \sigma_C \circ \pi_i$. Similarly, when $q$ is even, again according to the first part of Theorem~\ref{thm:aut}, there exists $C\in \fqt^*$, $i\in [k(t-1)]$ and $A\in\fqt$ such that $\Phi\big((X,x)\big)=\big(C^2X^{p^i}+A,N(C)x^{p^i}\big)$ and hence $\Phi=\mu_A\circ  \sigma_C \circ \pi_i$. This shows that these groups generate $\Aut(\NG(q,t))$, i.e.
\begin{equation*}
\Aut(\NG(q,t))=\left\{\begin{array}{ll}
\Aut_S\circ\Aut_M\circ \Aut_F & \text{ if $q$ is odd}\\
\Aut_L\circ\Aut_M\circ \Aut_F & \text{ if $q$ is even.}
\end{array}
\right.
\end{equation*}
To prove the appropriate group structure suppose first that $q$ is odd and consider the term $\Aut_S\circ\Aut_M$. If $t-1$ is also odd, then $\omega_{-1}=\sigma_{-1}$, hence $\Aut_S\subseteq\Aut_M$ and so $\Aut_S\circ\Aut_M=\Aut_M$. If $t-1$ is even, then $\Aut_S\cap\Aut_M=\{id\}$, elements from the two parts clearly commute and $\Aut_S\circ\Aut_M$ is also a subgroup of $\Aut(\NG(q,t))$, hence $\Aut_S\circ\Aut_M=\Aut_S\times\Aut_M$.

To add $\Aut_F$ we apply \Cref{lemma:ProductCondition} with $G=\Aut(\NG(q,t))$, $N=\Aut_s\circ\Aut_M$ and $H=\Aut_F$. For this we first need to check that $\displaystyle (\omega_{\ep}\circ\sigma_C)^{\pi_i}\in \Aut_S\circ\Aut_M$ for every $\ep\in \{-1,+1\}$, $C\in \fqt^*$ and $i\in [k(t-1)]$:
\begin{equation*}
(\omega_\ep\circ \sigma_C)^{\pi_i}=\pi_i^{-1}\circ (\omega_\ep\circ \sigma_C)\circ \pi_i=\pi_{k(t-1)-i}\circ\omega_\ep\circ \sigma_C\circ \pi_i=\omega_{\ep'}\circ \sigma_{C'}\in \Aut_S\circ\Aut_M,
\end{equation*}
where $\ep'=\ep^{p^{k(t-1)-i}}$ and $C'=C^{p^{l(t-1)-i}}$. We clearly also have that $(\Aut_S\circ \Aut_M)\cap \Aut_M=\{id\}$, so \Cref{lemma:ProductCondition} implies that $\Aut_S\circ \Aut_M$ is a normal subgroup of $\Aut_S\circ\Aut_M\circ \Aut_F=\Aut(\NG(q,t))$ and
\begin{align*}
\Aut(\NG(q,t))&=\Aut_S\circ\Aut_M\circ \Aut_F=(\Aut_S\circ\Aut_M)\rtimes \Aut_F\\
&=\left\{
\begin{array}{ll}
\Aut_M\rtimes \Aut_F= Z_{q^{t-1}-1}\rtimes Z_{k(t-1)} & \text{if $t-1$ is odd}\\
\left(\Aut_S\times \Aut_M\right)\rtimes \Aut_F=\left(Z_2\times Z_{\frac{q^{t-1}-1}{2}}\right)\rtimes Z_{k(t-1)} & \text{if $t-1$ is even.}
\end{array}
\right.
\end{align*}
Now suppose $q$ is even and consider first $\Aut_L\circ\Aut_M$. We again apply \Cref{lemma:ProductCondition}, now with $G=\Aut(\NG(q,t))$, $N=\Aut_L$ and $H=\Aut_M$. First we check that $\mu_A^{\sigma_C}\in \Aut_L$ for every $A\in \fqt$ and $C\in \fqt^*$:
\begin{equation*}
\mu_A^{\sigma_C}=\sigma_C^{-1}\circ \mu_A \circ\sigma_C=\sigma_{C^{-1}}\circ \mu_A \circ\sigma_C=\mu_{(C^{-1})^2A}\in \Aut_L.
\end{equation*}
We clearly also have $\Aut_L\cap\Aut_M=\{id\}$, so by \Cref{lemma:ProductCondition} $\Aut_L\circ \Aut_M$ is a subgroup of $\Aut(\NG(q,t))$, $\Aut_L$ is normal subgroup of it and we have
\begin{equation*}
\Aut_L\circ \Aut_M=\Aut_L\rtimes\Aut_M.
\end{equation*}
To obtain the whole of $\Aut(\NG(q,t))$, we apply \Cref{lemma:ProductCondition} one last time, now with $G=\Aut(\NG(q,t))$, $N=\Aut_L\circ \Aut_M=\Aut_L\rtimes \Aut_M$ and $H=\Aut_F$. We start by checking that $(\mu_A\circ\sigma_C)^{\pi_i}\in \Aut_L\circ\Aut_M$ for every $A\in \fqt$, $C\in \fqt^*$ and $i\in [k(t-1)]$: 
\begin{equation*}
(\mu_A\circ\sigma_C)^{\pi_i}=\pi_i^{-1}\circ \mu_A\circ\sigma_C\circ\pi_i=\pi_{k(t-1)-i}\circ \mu_A\circ\sigma_C\circ\pi_i=\mu_{A'}\circ\omega_{C'}\in \Aut_L\circ\Aut_M,
\end{equation*}
where $A'=A^{p^{k(t-1)-i}}$ and $C'=C^{p^{l(t-1)-i}}$. We clearly also have that $(\Aut_L\circ \Aut_M)\cap \Aut_F=\{id\}$, so using \Cref{lemma:ProductCondition} we get that $\Aut_L\circ \Aut_M$ is a normal subgroup of $\Aut_L\circ\Aut_M\circ \Aut_F=\Aut(\NG(q,t))$ and
\begin{align*}
\Aut(\NG(q,t))&=\Aut_L\circ\Aut_M\circ \Aut_F=(\Aut_L\circ\Aut_M)\rtimes \Aut_F\\
&=\left(\Aut_L\rtimes \Aut_M\right)\rtimes \Aut_F=\left((Z_p)^{k(t-1)}\rtimes Z_{q^{t-1}-1}\right)\rtimes Z_{k(t-1)}.
\end{align*}

\end{proof}

\section{Concluding remarks}\label{sec:conclusion}

{\bf Common neighbourhoods.} Recall that in Theorem~\ref{thm:main}(b)
we had to assume that $q$ is odd. We note that an analoguous result
can be shown for even characteristic as well. Namely, it holds that
$\deg(T)=q^{t-3} + O(q^{t-3.5})$ for all but $o(n^3)$ triples $T$ in
$\NG(q,t)$ with $q=2^k$ and $t\geq 4$. Furthermore the exceptional
cases can also be characterized. The main idea of the proof for odd
characteristic can be adapted, but the technicalities become different.
Together with Theorem~\ref{thm:main}(b) this extension settles the
question about common neighbourhoods of triples of vertices
completely. Based on computer calculations we conjecture that the
analogous ``$\ell$-wise independence'' phenomenon occurs for larger sets of vertices as well.
\begin{conjecture}\label{conj:neighb}
For any prime power $q$ and integers $4\leq \ell<t$ all but $o(n^\ell)$ sets
of $\ell$ vertices in $\NG(q,t)$ have $(1+o(1))q^{t-\ell}$ common neighbours.
\end{conjecture}

\vspace{0.5cm}

\noindent {\bf Complete bipartite graphs in projective norm graphs.}
As already discussed in the introduction, it is a fundamental problem
to determine for $t\geq 4$ the value of $s_t$, the largest integer
such that $\NG(q,t)$ contains $H=K_{t,s_t}$ for every large enough
prime power $q$. Note that Theorem~\ref{thm:K46}, because of the
annoying missing cases of characteristic $2$ and $3$, does not yet imply
$s_4=6$, but computer calculations strongly suggest this being the case. For larger values of $t$ the question remains widely open. 

\vspace{0.5cm}

\noindent {\bf Quasirandomness.} In Section~\ref{sec:applications} we
proved that if $q$ is an odd prime and $t\geq 4$ an integer then
$\NG(q,t)$ is $H$-quasirandom whenever $H$ is a fixed simple
$3$-degenerate graph. The extension of Theorem~\ref{thm:main}(b) 
also implies this to even $q$. A positive answer to
Conjecture~\ref{conj:neighb} would directly result in a generalization
of Theorem~\ref{thm:subgraph} stating that for any prime power $q$ and
integer $t\geq 3$ the projective norm graph $\NG(q,t)$ is $H$-quasirandom
for every fixed simple $(t-1)$-degenerate graph $H$.

It would be also interesting to study what can we say beyond the scope
of Conjecture~\ref{conj:neighb}, about the containment of an any fixed
graph.  Especially interesting would be the cases of cliques.
The so-called clique-graph of the projective norm graphs were
expliciteley used by Alon and Pudlak~\cite{AP} for their constructions
for the asymmetric Ramsey problem. 
They lower bound the clique number $\omega(NG(q,t))$ by the
Expander Mixing Lemma, which is probably far from being tight. In this paper we go
beyond that and show not only the existence of $K_4$, but also the
$K_4$-quasirandomness of $\NG(q,t)$ for $t\geq 4$. We are, however, still very far from the understanding of the behaviour of the clique number.
Besides its exact determination there are
several other intriguing directions. 
We think that once a ``nice'' fixed graph $H$ is contained in the projective
norm graph for every large enough $q$, then there are the ``right''
number of copies of it.
\begin{conjecture}\ \\
(i) For every $2\leq t\leq s\leq s_t$ the projective norm graph $\NG(q,t)$ is
$K_{t,s}$-quasirandom. \\
(ii) If $s\leq \omega (\NG(q,t))$ for every large enough $q$, then $\NG(q,t)$
is $K_s$-quasirandom.
\end{conjecture}
Finally, there is very little known about whether there are any
characteristic-specific
subgraphs. We do not know whether there is any fixed
graph $H$ which is contained in projective norm graphs for some
chracteristic $p_1$, but it is not contained in them for some other
charateristic $p_2$. 

\section{Appendix}

\subsection{Finite fields}

For a prime power $q=p^k$ let $\mathbb{F}_q$ and $\mathbb{F}_q^{*}$ denote the finite field of $q$ elements and its multiplicative group, respectively. For the sake of completeness we first recall some basic facts about finite fields. For proofs and details the interested reader may consult e.g. \cite{LN}.

\begin{lemma} \label{lemma:FiniteFieldBasics} Let $p$ be a prime,
  $k\in \mathbb{N}^+$, and $q=p^k$. Then
\ \begin{enumerate}[a)]
\item $\zb F_q$ exists and is unique up to isomorphism.
\item The multiplicative group $\zb F_{q}^*$ is cyclic, i.e. $\zb F_{q}^*\isom \mathbb{Z}_{q-1}$ .
\item As an additive group, $\zb F_{p^k}$ is isomorphic to $\kl{\zb{Z}_p}^k$.
\item The subfields of $\zb F_{p^k}$ are exactly those finite fields $\zb F_{p^s}$ for which $s\mid kl$.
\item The map $x\mapsto x^p$ is an automorphism of $\zb F_{p^k}$, and is called the Frobenius automorphism.
\item The automorphism group $\Aut(\zb F_{p^k})$ is generated by the Frobenius automorphism, i.e. any field automorphism of $\zb F_{p^k}$ is of the form $x \to x^{p^i}$ for some $i\in [k]$.
\item For $s\mid k$ the automorphism $x\mapsto x^{p^s}$ fixes the subfield $\zb F_{p^s}$.  \label{lemma:FiniteFieldBasics:Fix}
\end{enumerate}
\end{lemma}

For a prime power $q$ and $t\in \mathbb{N}^+$ the norm map $N_{q,t}: \mathbb{F}_{q^t} \rightarrow \fq$ is defined as
\begin{equation*}
N_{q,t}(X)=X\cdots X^q\cdots X^{q^{t-1}}.
\end{equation*}
In most cases $q$ and $t$ will be clear from the context (the
parameters of $\NG(q,t)$ at hand), and usually we will simply write $N(X)$ instead of $N_{q,t-1}(X)$.
The following lemma summarizes some important properties of the norm map $N=N_{q,t-1}$.

\begin{lemma} \label{lemma:NormProperties} Let $q$ be a prime power
  and $t\geq 2$ an integer. Then for $N=N_{q,t-1}$ we have the following.
\	\begin{enumerate}[a)]
\item $N(A)\in \mathbb{F}_q$ for every $A\in \mathbb{F}_{q^{t-1}}$ and $N(X)=0 \gdw X=0$.
\item The restriction of $N$ to the multiplicative group $\fqt^*$ is a
  group homomorphism onto $\fq^*$. 
\item $N(x)=x^{t-1}$ for every $x\in\zb F_q$.
\item If $X$ is a generator of $\zb F_{q^{t-1}}^*$ then $N(X)$ is a generator of $\fqu$.
\item $X\in\zb F_{q^{t-1}}^*$ is a square in $\zb F_{q^{t-1}}^*$ if and only if $N(X)$ is a square in $\fqu$.
\item $\anz{N^{-1}(x)} = \frac{q^{t-1}-1}{q-1}$ for every $x\in\fqu$.	
\end{enumerate} 
\end{lemma}

\subsection{Direct and semidirect products of groups}

Here we briefly recall the definitions of direct and semidirect product of groups and state a lemma we will be using when proving results about the group structure of the automorphism group of $\NG(q,t)$. For more details we refer to Chapters 4 and 7 of \cite{R}.

For a group $G$ and an element $a\in G$, the map defined by $x\mapsto
x^a=a\inv xa$ is a group automorphism of $G$, called \emph{conjugation
  by $a$}. A subgroup $N\leq G$ is called \emph{normal}, denoted by
$N\triangleleft G$, if $N^a\subset N$ for all $a\in G$. For two
subsets $S_1, S_2 \subseteq G$, we write  $S_1\cdot S_2 = \{ gh : g\in S_1, h\in S_2\}$. Now $G$ is said to be the \emph{internal direct product} of the subgroups $N_1$ and $N_2$, denoted by $G=N_1\times N_2$, if $G=N_1\cdot N_2$, $N_1\cap N_2=\men{1_G}$ and both $N_1$ and $N_2$ are normal subgroups of $G$. 

A natural way of generalizing the inner direct product is to weaken the restriction on the normality of the subgroups. More precisely a group $G$ is said to be the \emph{internal semidirect product} of the subgroups $N$ and $H$, denoted by $G=N\rtimes H$, if $G=N\cdot H$, $N\cap H=\men{1_G}$ and $N$ is a normal subgroup. 

The following lemma (whose proof is an easy exercise) will be used several times to prove that a given group is the inner semidirect product of two of its subgroups.

\begin{lemma} \label{lemma:ProductCondition}
Let $G$ be a group and $N,H$ subgroups such that $N^h\subseteq N$ for every $h\in H$. Then $N\cdot H$ is a subgroup of $G$ and $N\triangleleft N\cdot H$. If $N\cap H=\men{1_G}$, then $N\cdot H=N\rtimes H$.
\end{lemma}

\subsection{Characters}

Next we recall some basic facts about characters of finite groups that we will be using in later sections. For proofs and further results the interested reader may consult e.g. \cite[Chapter 5.1]{LN}.

For a finite abelian group $G$ a group homomorphism $\chi$ from $G$ to the multiplicative group $\zb C^*$ of complex numbers is called a \emph{character} of $G$. The smallest integer $m\inn$ such that $\chi^m\equiv 1$ is called the \emph{order} of $\chi$. When $\zb F$ is a finite field and $\chi$ is a character of $\zb F^*$, it is convenient to extend $\chi$ to $0 \in \zb F$ by setting $\chi(0)=0$. Abusing terminology, here we identify $\chi$ with this extension and call the extension itself a character of the field $\zb F$. A nice property of a character $\chi$ of order $m>1$ of a group $G$ is that $\fsum{a\in G}{}\chi(a)=0$. The following result of Weil (see \cite[Thm 5.41]{LN}) states that for characters of finite fields the above result can be generalized.

\begin{theorem}\label{thm:Weil}
Let $\chi$ be an order $m$ character of $\fq$ and $f\in\fq[X]$ a univariate polynomial of degree $d\geq 1$ which is not of the form $cg^m$ for some $c\in\fqu$ and $g\in\fq[X]$. Then
\begin{equation*}
\anz{\fsum{a\in\fq}{} \chi(f(a))} \leq (d-1)\sqrt{q}.
\end{equation*}
If $f=cg^m$ for some $c\in\fqu$ and $g\in\fq[X]$, then 
\begin{equation*}
\fsum{a\in\fq}{} \chi(f(a)) = (q-r)  \cdot \chi(c),
\end{equation*}
where $r$ is the number of distinct roots of $g$ over $\fq$.
\end{theorem}

In further sections we will be interested in one particular type of character. If $G$ is a finite cyclic group then its \emph{quadratic character} $\eta_G$ is defined as
\begin{equation*}
\eta_G(a) = \absw{ 1 & \text{ if } \exists b\in G: b^2 = a \\ -1 & \text{ otherwise} }.
\end{equation*}
$\eta_G$ is indeed a character of $G$ and is of order $1$ or $2$, depending on whether $\anz{G}$ is odd or even. As the multiplicative group of any finite field $\zb F$ is cyclic, there is also an associated quadratic character $\eta_\zb F$. Usually we extend $\eta_\zb F$ to the whole field by setting $\eta_\zb F(0)=0$.  Among others, it can be used to express the number of roots of a quadratic polynomial.

\begin{lemma} \label{lemma:rootsQuadratic}
Let $\zb F$ be a finite field with $\Char(\zb F)\neq 2$ and $p\in \zb F[X]$ a quadratic polynomial with discriminant $D\in \zb F$. Then $p$ has $1+\eta_\zb F(D)$ distinct roots in $\zb F$.
\end{lemma}

\subsection{Proof of Theorem~\ref{thm:main}(b)}

\begin{claim}\label{coeffclaim}
For $c_1,c_2\in \mathbb{F}_q^*$ the polynomial 
\begin{equation*}
L_{c_1,c_2}(b)=b^4+2(c_1-c_2-1)b^3+\big((1+c_1-c_2)^2-6c_1\big)b^2+2c_1(1-c_1-c_2)b+c_1^2
\end{equation*}
is of the form $(b^2+\alpha_1b+\alpha_0)^2$ for some $\alpha_1,\alpha_0\in \mathbb{F}_0$ if and only if $(c_1,c_2)=(1,-1)$, in which case we have $(\alpha_1,\alpha_0)=(1,1)$.
\end{claim}
\begin{proof}
First suppose $(c_1,c_2)=(1,-1)$. Then $L_{1,-1}(b)=b^4+2b^3+3b^2+2b+1=(b^2+b+1)^2$.

For the other direction suppose 
\begin{equation*}
L_{c_1,c_2}(b)=(b^2+\alpha_1b+\alpha_0)^2=b^4+2\alpha_1b^3+(\alpha_1^2+2\alpha_0)+2\alpha_1\alpha_1b+\alpha_0^2
\end{equation*}
for some $\alpha_1,\alpha_0\in \mathbb{F}_0$. By comparing coefficients we arrive at the system
\begin{align}
2(c_1-c_1-1)&=2\alpha_1\label{compare11}\\
(1+c_1-c_2)^2-6c_1&=\alpha_1^2+2\alpha_0\label{compare12}\\
2c_1(1-c_1-c_2)&=2\alpha_1\alpha_0\label{compare13}\\
c_1^2&=\alpha_0^2. \label{compare14}
\end{align}
(\ref{compare14}) just means that we either have $\alpha_0=c_1$ or $\alpha_{0}=-c_1$.

\textbf{Case $\alpha_0=c_1$}: After substitution, using (\ref{compare11}) and (\ref{compare13}) we obtain
\begin{equation*}
c_1-c_2-1=\alpha_1=1-c_1-c_2,
\end{equation*}
and accordingly $\alpha_0=c_1=1$ and $\alpha_1=-c_2$. Substituting all this into (\ref{compare12}) we obtain
\begin{equation*}
(2-c_2)^2-6=c_2^2+2,
\end{equation*}
and accordingly $c_2=-1$ and $\alpha_1=1$.

\textbf{Case $\alpha_0=-c_1$}: After substitution, using (\ref{compare11}) and (\ref{compare13}) we obtain
\begin{equation*}
c_1-c_2-1=\alpha_1=-1+c_1+c_2,
\end{equation*}
and accordingly $c_2=0$, which is impossible. 
\end{proof}

\subsection{Proof of \cref{lemma:SpecialQuadruples}}

\begin{claim}\label{linear_eqs}
Let $q$ be a prime power and $A\neq B\in \fqq{3}$ such that $N(A)=N(B)=1$. Then
$X_1=X_1(A,B)=\displaystyle\frac{-A^{q+1}B+AB^{q+1}}{A^{q+1}-B^{q+1}}$ and $X_2=X_2(A,B)=\displaystyle\frac{A^{q+1}-B^{q+1}}{-A^q+B^q}$ are solutions to (\ref{f1})
 and (\ref{f4})
 respectively.
\end{claim}
\begin{proof} For $X_1$ first consider
\begin{align*}
h(X_1,A)&=X_1^{q+1}+X_1^qA+A^{q+1}=X_1^q(X_1+A)+A^{q+1}\\
&=\left(\frac{-A^{q+1}B+AB^{q+1}}{A^{q+1}-B^{q+1}}\right)^{q}\left(\frac{-A^{q+1}B+AB^{q+1}}{A^{q+1}-B^{q+1}}+A\right)+A^{q+1}\\
&=\left(\frac{-A^{q^2+q}B^q+A^qB^{q^2+q}}{A^{q^2+q}-B^{q^2+q}}\right)\left(\frac{-A^{q+1}B+AB^{q+1}+A^{q+2}-AB^{q+1}}{A^{q+1}-B^{q+1}}\right)+A^{q+1}\\
&=\frac{\big(-A^{q^2+q}B^q+A^qB^{q^2+q}\big)A^{q+1}\big(A-B)}{\big(A^{q^2+q}-B^{q^2+q}\big)\big(A^{q+1}-B^{q+1}\big)}+A^{q+1}.
\end{align*}
Putting $\displaystyle m(A,B)=\frac{A^{q+1}}{\big(A^{q^2+q}-B^{q^2+q}\big)\big(A^{q+1}-B^{q+1}\big)}$ we get
\begin{align*}
h(X_1,A)=&m(A,B)\Big(\big(-A^{q^2+q}B^q+A^qB^{q^2+q}\big)\big(A-B)+\big(A^{q^2+q}-B^{q^2+q}\big)\big(A^{q+1}-B^{q+1}\big)\Big)\\
=&m(A,B)\Big(-A^{q^2+q+1}B^q+A^{q^2+q}B^{q+1}+A^{q+1}B^{q^2+q}-A^qB^{q^2+q+1}\\
&+A^{q^2+q+1}A^q-A^{q^2+q}B^{q+1}-A^{q+1}B^{q^2+q}+B^{q^2+q+1}B^q\Big)
\end{align*}
and so using $N(A)=A^{q^2+q+1}=1$ and $N(B)=B^{q^2+q+1}=1$ we have
\begin{align*}
h(X_1,A)=&m(A,B)\Big(-B^{q}+A^{q^2+q}B^{q+1}+A^{q+1}B^{q^2+q}-A^q+A^q-A^{q^2+q}B^{q+1}\\
&-A^{q+1}B^{q^2+q}+B^q\Big)=m(A,B)\cdot 0=0.
\end{align*}
By switching the roles of $A$ and $B$ we also have that $h(X_1(B,A),B)=0$. However $X_1(B,A)=X_1(A,B)$ and so
\begin{equation*}
h(X_1,B)=h(X_1(A,B),B)=h(X_1(B,A),B)=0,
\end{equation*}
which means that $X_1$ is in fact a solution to (\ref{f1}).

\noindent For $X_2$ first note that 
\begin{equation*}
X_2=\frac{A^{q+1}-B^{q+1}}{-A^q+B^q}=AB \frac{A^{q+1}-B^{q+1}}{-A^{q+1}B+AB^{q+1}}=\frac{AB}{X_1}.
\end{equation*}
However then
\begin{align*}
h(A,X_2)=h(A,\frac{AB}{X_1})=A^{q+1}+A^q\frac{AB}{X_1}+\left(\frac{AB}{X_1}\right)^{q+1}
=\frac{A^{q+1}}{X_1^{q+1}}h(X_1,B)=0
\end{align*}
and
\begin{align*}
h(B,X_2)=h(B,\frac{AB}{X_1})=B^{q+1}+B^q\frac{AB}{X_1}+\left(\frac{AB}{X_1}\right)^{q+1}
=\frac{B^{q+1}}{X_1^{q+1}}h(X_1,A)=0
\end{align*}
which means that $X_2$ is really a solution to (\ref{f4})
\end{proof}

\begin{claim}\label{quadratic_eqs}
Let $q$ be an odd prime power and $A,B\in \fqq{3}$ such that $N(A)=N(B)=1$. Then 
all $\fqq{3}$-solutions of (\ref{f2*})
and (\ref{f3*})
are solutions to (\ref{f2})
 and (\ref{f3}),
 respectively.
\end{claim}
\begin{proof}
First note that it is enough to prove that the $\fqq{3}$-solutions of (\ref{f2*})
are solutions to (\ref{f2})
as by swithching the roles of $A$ and $B$ (\ref{f2*})
transforms to (\ref{f3*})
and (\ref{f2})
to (\ref{f3})
. We start the proof by deriving, using $N(A)=A^{q^2+q+1}=1$ and $N(B)=B^{q^2+q+1}=1$, some useful identities.
\begin{align}\label{habq}
h(A,B)^q&=\big(A^{q+1}+A^qB+B^{q+1}\big)^q=A^{q^2+q}+A^{q^2}B^q+B^{q^2+q}\nonumber \\
&=A^{q^2+q}+A^{q^2}B^q+A^{q^2+q+1}B^{q^2+q}=\frac{A^{q^2}}{B}\big(A^qB+B^{q+1}+A^{q+1}B^{q^2+q+1}\big) \\
&=\frac{A^{q^2}}{B}\big(A^{q+1}+A^qB+B^{q+1}\big)=\frac{A^{q^2}}{B}h(A,B) \nonumber
\end{align}
By switching the roles of $A$ and $B$ we also obtain
\begin{equation}\label{hbaq}
h(B,A)^q=\frac{B^{q^2}}{A}h(B,A).
\end{equation}
Then for $D_1=D_1(A,B)$, using again the norm conditions together with (\ref{habq}) and (\ref{hbaq}), we have
\begin{align}\label{d1q}
D_1^q&=\big(h(B,A)^2-4 B^q A ~ h(A,B)\big)^q=\big(h(B,A)^q\big)^2-4^qB^{q^2}A^qh(A,B)^q\nonumber \\
&=\left(\frac{B^{q^2}}{A}h(B,A)\right)^2-4B^{q^2}A^q\frac{A^{q^2}}{B}h(A,B)=\frac{B^{2q^2}}{A^2}\Big(h(B,A)^2-4\frac{A^{q^2+q+1}A}{B^{q^2+1}}h(A,B)\Big) \\
&=\frac{B^{2q^2}}{A^2}\Big(h(B,A)^2-4\frac{AB^{q^2+q+1}}{B^{q^2+1}}h(A,B)\Big)=\frac{B^{2q^2}}{A^2}\big(h(B,A)^2-4 B^q A ~ h(A,B)\big)=\frac{B^{2q^2}}{A^2}D_1 \nonumber.
\end{align}
Now let $C_1\in \fqq{3}$ 
be a solution to (\ref{f2*})
. 
We intend to show that $C_1$ is also a solution to (\ref{f2})
, i.e. $h(C_1,A)=h(B,C_1)=0$. 
Clearly $C_1$ 
can be written as
\begin{equation*}
C_1=\frac{2 A^{q+1}-h(B,A)+ G_1}{2\cdot B^q}
\end{equation*}
where $G_1
\in \fqq{3}^*$ 
is such that  $G_1^2=D_1$. 
Note that $D_1\not=0$ because exactly one of $h(A,B)$ and $h(B,A)$ is 0,
hence $G_1\not=0$. Next we want to express $G_1^q$ 
in terms of $A$ and $B$.

Using (\ref{d1q}) we have
\begin{align*}
\left(\frac{G_1^q}{G_1}\right)^2=\frac{\big(G_1^2\big)^q}{G_1^2}=\frac{D_1^q}{D_1}=\frac{\frac{B^{2q^2}}{A^2}D_1}{D_1}=\frac{B^{2q^2}}{A^2}\text{ and hence }\frac{G_1^q}{G_1}=\pm\frac{B^{q^2}}{A}.
\end{align*}
However as $G_1\in \fqq{3}^*$ we have $N\left(\frac{G_1^q}{G_1}\right)=N(G_1^{q-1})=N(1)=1$ which excludes $\frac{G_1^q}{G_1}=-\frac{B^{q^2}}{A}$ as in this case by $N(A)=N(B)=1$ we would have $N\left(\frac{G_1^q}{G_1}\right)=N\left(-\frac{B^{q^2}}{A}\right)=-\frac{N(B)^{q^2}}{N(A)}=-1$. Accordingly
\begin{equation}\label{g1q}
G_1^q=\frac{B^{q^2}}{A}G_1.
\end{equation}
Then using the norm conditions together with (\ref{hbaq}) and (\ref{g1q}) we get
\begin{align}\label{c1q}
\displaystyle C_1^q&=\left(\frac{2 A^{q+1}-h(B,A)+ G_1}{2\cdot B^q}\right)^q=\frac{2^q A^{q^2+q}-h(B,A)^q+ G_1^q}{2^q\cdot B^{q^q}}\nonumber\\
&=\frac{2 A^{q^2+q}-\frac{B^{q^2}}{A}h(B,A)+ \frac{B^{q^2}}{A}G_1}{2\cdot B^{q^2}}=\frac{2 \frac{A^{q^2+q+1}B^{2q+1}}{A}-\frac{B^{q^2+q+1}B^q}{A}h(B,A)+ \frac{B^{q^2+q+1}B^q}{A}G_1}{2\cdot B^{q^2+q+1}B^q}\nonumber\\
&=\frac{B^q}{A}\left(\frac{2 A^{q+1}-h(B,A)+ G_1}{2\cdot B^q}\right)+\frac{B^{q+1}-A^{q+1}}{A}=\frac{B^q}{A}C_1+\frac{B^{q+1}-A^{q+1}}{A}.
\end{align}
Now we are ready to do the final steps, namely to substitute $C_1$ 
into the respective polynomials.

Using (\ref{c1q}) and the fact that $C_1$ is a root of (\ref{f2*})
we get
\begin{align*}
h(C_1,A)&=C_1^{q+1}+C_1^qA+A^{q+1}=C_1^q(C_1+A)+A^{q+1}\\
&=\left(\frac{B^q}{A}C_1+\frac{B^{q+1}-A^{q+1}}{A}\right)(C_1+A)+A^{q+1}\\
&=\frac{1}{A}\left(B^qC_1^2+(AB^q+B^{q+1}-A^{q+1})C_1+AB^{q+1}\right)=\frac{1}{A}\cdot 0=0
\end{align*}
and
\begin{align*}
h(B,C_1)&=B^{q+1}+B^qC_1+C_1^{q+1}=B^{q+1}+C_1(B^q+C_1^{q})\\
&=B^{q+1}+C_1\left(B^q+\frac{B^q}{A}C_1+\frac{B^{q+1}-A^{q+1}}{A}\right)\\
&=\frac{1}{A}\left(B^qC_1^2+(AB^q+B^{q+1}-A^{q+1})C_1+AB^{q+1}\right)=\frac{1}{A}\cdot 0=0,
\end{align*}
hence $C_1$ is indeed a root of (13)

\end{proof}

\end{document}